\documentclass{template_report}

\usepackage{amsmath,amsfonts,amssymb,amsthm}
\usepackage{xcolor}
\usepackage{enumitem,hyperref,bm}

\usepackage{pgfplots}
\usepackage{mathrsfs}
\usetikzlibrary{arrows}
\usetikzlibrary{arrows.meta}

\usepackage{tikz}
\usetikzlibrary{fit}

\usepackage{thmtools}

\newtheorem{lemma}{Lemma}[chapter]
\newtheorem{theorem}[lemma]{Theorem}

\newtheorem{proposition}[lemma]{Proposition}
\newtheorem{definition}[lemma]{Definition}
\newtheorem{corollary}[lemma]{Corollary}

\newtheorem{question}[lemma]{Question}

\theoremstyle{definition}
\newtheorem{remark}[lemma]{Remark}
\newtheorem{example}[lemma]{Example}
\newtheorem{exercise}[lemma]{Exercise}

\newcommand{\ds}{\displaystyle}

\newcommand{\eps}{\varepsilon}
\newcommand{\N}{\mathbb{N}}
\newcommand{\Z}{\mathbb{Z}}
\newcommand{\Q}{\mathbb{Q}}
\newcommand{\R}{\mathbb{R}}
\newcommand{\A}{\mathcal{A}}

\newcommand{\pref}{\mathrm{pref}}
\newcommand{\card}{\mathrm{card}}
\newcommand{\lan}{\mathcal{L}}
\newcommand{\G}{\mathcal{G}}

\newcommand{\Span}{\mathrm{Span}}
\newcommand{\rank}{\mathrm{rk}}
\newcommand{\ext}{\mathrm{Ext}}

\title{Infinite Words with very Low Factor Complexity: \\an introduction to Combinatorics on Words}

\author[1,2]{M\'elodie Andrieu}

\affil[1]{Laboratoire de Math\'ematiques Pures et Appliqu\'ees Joseph Liouville, Universit\'e du Littoral C\^ote d'Opale, France.}

\affil[2]{Centro de Modelamiento Matem\'atico, Universidad de Chile and IRL-CNRS 2807,  Chile. \par\email {melodie.andrieu@univ-littoral.fr}}

\documenttitle{Infinite Words with very Low Complexity}

\documentauthor{M. Andrieu}
\date{December 2025}

\begin{document}
	
	\maketitle
	
	\begin{abstract}

		These lecture notes provide an introduction to combinatorics on words and its interactions with dynamics, algebra, and arithmetic.

		The central theme is the notion of low factor complexity for infinite words. We investigate the following guiding questions: What is the minimal complexity of a non-trivial infinite word over a binary, ternary, or more generally finite alphabet? How should ``non-triviality'' be formalized? Which words achieve this minimal complexity? Are there many? Are they interesting?

		In exploring these questions, we  introduce classical objects and tools from combinatorics on words such as \emph{Sturmian words} and \emph{Rauzy graphs}, as well as little-known and new results. 
		In particular, the third chapter is devoted to a theorem by R. Tijdeman from 1999, which generalizes a seminal result of M. Morse and G. Hedlund from 1938, and of which we provide a conceptually new, algebraic proof by J. Cassaigne and the author from 2022. 
		
		The first chapter requires no prior knowledge in combinatorics on words and symbolic dynamics. It is intended for graduate students. The second and third chapters are intended both for graduate students and researchers familiar with Sturmian words and their generalizations.  
		\medskip
		
		\emph{This text is based on a lecture delivered during the 7th edition of the international conference Dyadisc, held in Valpara\'iso, Chile, in December 2024. The new results it contains have also been presented in various seminars since 2023.}

	\end{abstract}
	
	%
	%
	
	\setcounter{page}{1}
	\tableofcontents
	
	\cleardoublepage
	
	\chapter*{Overview}
	\addcontentsline{toc}{chapter}{Overview}

	\noindent Let $d\geq 1$ be an integer. An \emph{infinite $d$-ary word} is a  sequence $w \in  \{1,\ldots,d\}^{\N}$. The \emph{complexity} of $w$ is the function that counts, for every integer $n$, the number of distinct subwords of length $n$ that appear in $w$.
	The aim of this course is to answer the following two questions:
	
	1) What is the \emph{minimal} complexity of \emph{interesting} $d$-ary words? 
	
	2) Which words achieve this minimal complexity?
	
	\noindent Of course, answering these questions will require a rigorous and legitimate definition for \emph{minimal} and \emph{interesting}.
	
	\bigskip

	For binary words, the situation has been perfectly understood since the 1940s: \emph{Sturmian words} are exactly the \emph{interesting} (which in this case means \emph{non-eventually periodic}) words with minimal complexity. Chapter~\ref{sec:MH38_Sturmian} is a self-contained introduction to Sturmian words. We introduce them from a purely combinatorial perspective, and show how they relate to continued fractions, and to some dynamical systems. In particular, we emphasize why the rational independence of their letter frequencies is a central property of Sturmian words. In Chapter~\ref{sec:Generalization_MH38}, we explore how both the seminal theorem of M. Morse and G. Hedlund from 1938 and Sturmian words can be generalized to $d$-ary alphabets, when $d\geq3$. We will see that, in this context, ``interesting''  can no longer mean  ``non-eventually periodic'', but should instead become ``with rationally independent letter frequencies''. We conclude the chapter by discussing a theorem  by R. Tijdeman from 1999 (independently rediscovered by J. Cassaigne and  the author in 2022), which gives the minimal complexity of $d$-ary words with rationally independent letter frequencies. Finally, in Chapter~\ref{sec:Proof_Tijdeman}, we provide a conceptually new, algebraic proof of Tijdeman's theorem, and develop its consequences on the combinatorial structure of $d$-ary words with rationally independent letter frequencies and minimal complexity.

\medskip 

\noindent The three chapters can be read independently.

\vspace{5cm}

	\paragraph{Acknowledgements} The author is indebted to L\'eo Vivion for his inestimable help in the preparation of these lecture notes.
%

	
	\cleardoublepage
	
	\chapter{A seminal theorem by Morse and Hedlund (1938)}\label{sec:MH38_Sturmian}

	\section{Preliminaries: infinite words, complexity}

	Let $u$ be a finite word and $w$ be an \emph{infinite word} (by definition: a sequence in $\A^\N$ where $\A$ is a finite set called an \emph{alphabet}). We say that \emph{$u$ is a factor of length $n$ of $w$}, which we denote by $u\in\lan_n(w)$, if $u$ occurs in $w$ as a block of $n$ consecutive letters. For example, $\texttt{art}$ is a factor of length 3 of the French word $\texttt{montmartre}$ ($\texttt{art}\in\lan_3(\texttt{montmartre})$).
	
	\begin{definition}
		The \emph{complexity} of an infinite word $w$ is the function that counts, for each integer $n$, the number of factors of length $n$ in $w$:
		\[
		\begin{array}{rccl}
			p_w: & \N & \longrightarrow & \N\\
			& n & \longmapsto & \card\big(\lan_n(w)\big).
		\end{array}
		\]
	\end{definition}
	
	\noindent (Note that, by convention, we always have $p_w(0)=1$: every word  admits exactly one factor of length $0$, the ``empty word'', denoted by $\epsilon$.)
	
	\medskip
	We let the reader check the following examples.
	
	\begin{example}
		(i) If $w_1=1^\omega=111111...$ is the constant unary word, then $p_{w_1}(n)=1$ for all $n\in\N$. Indeed, the unique factor of length $n$ of $w_1$ is $1...1$ (that is, the letter $1$  repeated $n$ times). Note that $w_1$ has the absolute minimal complexity.\\
		(ii) If $w_2=(12)^\omega=12121212...$ is the periodic word with period $12$, then $p_{w_2}(n)=2$ for all $n\in\N_{\geq 1}$. Indeed, $w_2$ admits exactly two factors of length $n$: one that starts with the letter $1$, and one that starts with the letter $2$. \\
		(iii) If $w_3=(1233)^\omega=1233\cdot 1233\cdot 1233...$ is the periodic word with period $1233$, then $p_{w_3}(1)=3$ and $p_{w_3}(n)=4$ for all $n\in\N_{\geq 2}$. Indeed, $w_3$ is written with exactly $3$ letters, and every factor of length $n\geq 2$ is solely determined by its first two letters: $12$, $23$, $33$, or $31$.\\
		(iv) The decimal \emph{Champernowne word}, which is obtained by concatenating the base $10$ expansions of all integers:
		\[
		w_C=0123456789101112131415161718192021...
		\]
		has complexity $p_{w_C}(n)=10^n$, $\forall\, n\in\N$. Indeed, every  finite $10$-ary word appears in $w_C$. (Exercise: Where is the first occurrence of the factor $02$?) This complexity is maximal for infinite $10$-ary words.\\
		(v) It is believed, but far from being proved, that the decimal expansion of the golden ratio  also has  maximal complexity. This claim is a particular case of a much larger conjecture, attributed to E. Borel in 1950, which asserts that every irrational algebraic number is ``normal in every base'' \cite{Bor50}.
	\end{example}
	
	As the name suggests, the complexity quantifies how complicated an infinite word is. This combinatorial function is to be understood as a  refinement of the dynamical notion of entropy, which measures how much the knowledge of a small portion of a trajectory gives information about its possible long-term behavior. The next section establishes that, on this scale, periodic words are the simplest words.

	\section{The seminal theorem by Morse and Hedlund}
	
	The  next theorem appeared in the first article of M. Morse and G. Hedlund \cite{MH38}; it is one of the foundational results of Combinatorics on Words.
	
	\begin{theorem}[M. Morse and G. Hedlund, 1938] \label{MH38} 
		The following assertions are equivalent:
		\begin{enumerate}
			\vspace{-0.2cm}\item[\emph{(i)}] The infinite word $w$ is eventually periodic;
			\vspace{-0.2cm}\item[\emph{(ii)}] Its complexity $p_w$ is bounded;
			\vspace{-0.2cm}\item[\emph{(iii)}] There exists $n\in\N$ such that $p_w(n)\leq n$.
		\end{enumerate}
	\end{theorem}
	
	\begin{proof} (i) $\Rightarrow$ (ii) An eventually periodic word $w=u\cdot p^\omega$ admits at most one new factor of length $n$ (whatever $n$ is) starting at each position of its preperiod $u$ or its period $p$. Therefore, its complexity $p_w$ is bounded above by the sum of the lengths of its preperiod and period: $|u|+|p|$.\\
		(ii) $\Rightarrow$ (iii) This is trivial: it suffices to take $n=M$, where $M\in\N$ is an upper bound for the complexity function.\\
		(iii) $\Rightarrow$ (i) Since $p_w(0)=1$ (it was our convention), $p_w(n)\leq n$, and since $n\mapsto p_w(n)$ is non-decreasing (indeed, every factor of length $n$ of $w$ can always be extended into a factor of length $n+1$), there must exist $n_0\in\{1,\ldots,n-1\}$ such that $p_w(n_0)=p_w(n_0+1)$. This equality implies that, if $u$ is a factor of length $n_0$, then all its occurrences in $w$ are always followed by the same letter. (Otherwise, $u$ would give rise to two factors of length $n_0+1$, implying that $p_w(n_0+1)>p_w(n_0)$.) Furthermore, among the $p_w(n_0)$ factors of length $n_0$ occurring in $w$, at least one, let us call it again $u$, must occur infinitely often. Thus, there exists $v\in\lan(w)$ such that $uvu\in\lan(w)$ (by $\lan(w)$ we denote the set of all factors of $w$, regardless of their length). Since every factor of length $n_0$ of $w$ is always followed by the same letter, we deduce that the factor $u$ is always followed by the factor $vu$. This shows that there exists $p\in\lan(w)$ such that $w=pu(vu)^\omega$, \emph{i.e.}, $w$ is eventually periodic.
	\end{proof}

	\begin{remark}\label{rk:MH38}
		It follows from this proof that a word $w$ is eventually periodic if and only if its complexity function is \emph{locally constant}: there exists a length $m$ for which $p_w(m+1)=p_w(m)$, if and only if, its complexity is  eventually constant:  for all $n\geq m$, $p_w(n)=p_w(m)$.
	\end{remark}
	
	The contraposition of (iii) $\Rightarrow$ (i) states that if $w$ is not eventually periodic, then its complexity satisfies $p_w(n)\geq n+1$ for all $n\in\N$. This leads to the following natural question.
	
	\begin{question}
		Does there exist an infinite word $w$ whose complexity is exactly $p_w(n)=n+1$ for all $n\in\N$?
	\end{question}
	
	\noindent As we shall see, the answer to this question is `yes'. This positive answer will raise three new questions:
	
	1. How many infinite words satisfy this complexity condition?
	
	2. What do they look like? Can we construct them effectively/efficiently? 
	
	3. Are they interesting? 
	
	\noindent The remainder of this chapter is devoted to answering these questions. As a spoiler, the answers will be: `uncountably many, essentially one for each irrational number', `yes', `yes', and `YES!'. 
	
	\begin{definition}
		The infinite words $w$ with complexity $p_w(n)=n+1$ for all $n\in\N$ are called \emph{Sturmian words}.
	\end{definition}

	\begin{example}
		Although we have not yet proven that Sturmian words exist (this will be done later), here is a first example, called \emph{the Fibonacci word}. 
		\[w_{fibo} = 1211212112112121121211211212112112121121211211212112121...\]
		In the next sections, we will see how this word is constructed, and why it is called ``Fibonacci''.
	\end{example}
	
	\begin{exercise}
		Prove that, if an infinite word satisfies $p_w(n)=n+1$ for infinitely many $n\in\N$, then it is Sturmian.
	\end{exercise}
	
	\section{The class of Sturmian words}
	
	The study of Sturmian words was initiated in 1940 by M. Morse and G. Hedlund in their second article \cite{MH40}. They have been, and continue to be, extensively studied. This section is intended to be a self-contained introduction to Sturmian words. However, it is far from being exhaustive (for example, abelian properties are deliberately missing). We refer the interested reader to \cite[Chapter~2]{Loth02} for complements. \bigskip
	
	\emph{Outline of the section. } In the next two subsections, we assume that Sturmian words do exist and study their properties. We will see that, if they exist, they must be intimately connected to continued fractions. In the third subsection, we use continued fractions to design an effective and efficient way to construct them---thus finally proving that they \emph{indeed} exist. In the fourth and last subsection, we explain how Sturmian words can be dynamically and geometrically visualized.

	\subsection{Basic properties}
	
	The next proposition gathers elementary properties of Sturmian words that will be needed in the sequel.
	
	\begin{proposition}\label{prop:basics}
		Let $w$ be a Sturmian word. Then:
		\begin{itemize}
			\vspace{-0.25cm}\item[\emph{(i)}] $w$ is a binary word (hereafter, we set the alphabet to $\A=\{1,2\}$).
			\vspace{-0.25cm}\item[\emph{(ii)}] $w$ is ``recurrent'' (by definition: every factor of $w$ occurs infinitely many times).
			\vspace{-0.25cm}\item[\emph{(iii)}] If $S$ denotes the ``shift map'' (by definition: the map that erases the first letter of an infinite word; for instance,  $S(314159...)=14159...$), then $S(w)$ is also a Sturmian word.
			\vspace{-0.25cm}\item[\emph{(iv)}] For every $n\in\N$, there exists a unique factor $u\in\lan_n(w)$ such that both extensions $u1$ and $u2$ are factors of $w$ (definition: we  say that $u$ is ``right-special'').
		\end{itemize}
	\end{proposition}
	
	\begin{proof}
		(i) Trivial: since $p_w(1)=2$, $w$ uses exactly two letters.\\
		(ii) By contradiction, assume that $u\in\lan_n(w)$ does not occur infinitely many times; suppose, for example, that it does not occur after position $N\in\N$. Then, if we denote by $w':=S^N(w)$ the copy of $w$ after position $N$, we have $u\notin\lan_n(w')$, hence $p_{w'}(n)<p_w(n)=n+1$. This contradicts Theorem~\ref{MH38}, since $w'$ remains non-eventually periodic.\\
		(iii) Trivially, every factor of $S(w)$ is a factor of $w$. Since $w$ is recurrent, the converse is also true: every factor of $w$ is a factor of $S(w)$. Therefore, the complexity functions of $S(w)$ and $w$ are equal, and $S(w)$ is a Sturmian word.\\
		(iv) Since $w$ is a binary word, the quantity $p_w(n+1)-p_w(n)$ is exactly the number of right-special factors of length $n$ of $w$. For a Sturmian word $w$, this difference equals $1$, so for each length $n$, $w$ admits exactly one right-special factor of length $n$.
	\end{proof}
	
	\begin{exercise}
		Let $w$ be a Sturmian word. Prove that for every $n\in\N$, there exists a unique $v\in\lan_n(w)$ such that both left-extensions $1v$ and $2v$ are factors of $w$ (definition: we say that $v$ is ``left-special'').
	\end{exercise}

	\subsection{Interlude: Crash course on the continued fraction}\label{subsect:crash_courseCF}
	
	Before going further, let us recall what a continued fraction is, and why it is an extremely interesting object. (In the next subsection, we will explain how it connects to Sturmian words.)
	
	\medskip
	
	The continued fraction is a \emph{numeration system}, that is, a way to represent real numbers by a finite or infinite sequence of integers. The continued fraction expansion $(a_n)_n$ of a positive real number $x$ is classically denoted by
	\[
	x=[a_0;a_1,a_2,a_3,\cdots].
	\]
	For example, the continued fraction expansion of $17/6$ is $[2;1,5]$, the continued fraction expansion of the golden ratio $\varphi:=(1+\sqrt5)/2$ is $[1;1,1,1,1,1,1...]$ (the infinite sequence  constant, equal to $1$), while the continued fraction expansion of $\pi$ is an infinite sequence that starts with $[3; 7, 15, 1, 292, 1, 1, 1, 2,...]$.
	Contrary to the base expansion, the integers $a_n$ that appear in the continued fraction, which we call ``partial quotients'', are not necessarily bounded. 
	
	\medskip

	\noindent \emph{Question: How do we calculate the partial quotients?} 
	
	The partial quotients $(a_n)$ of $x \in \R_{>0}$ are exactly the successive \emph{integer} quotients that appear in the ``Euclidean algorithm'' applied to the pair of \emph{real numbers} $(x,1)$. For example, the first partial quotient $a_0$ of $x$ is the unique integer $q$ such that
	
	\[x=q\times 1+r,\]
	with $0\leq r<1$. (In other words, $a_0 = \lfloor x\rfloor$ and $r=\{x\}$ are the integer and fractional  parts of $x$.) Then, the second partial quotient $a_1$ is the integer quotient in the ``Euclidean division'' of $1$ by $r$ (which is also, equivalently, the integer quotient in the Euclidean division of $1/r$ by $1$), and so on. Algorithmically, we would write:

	\begin{equation*}\begin{array}{ll} \hspace{0.3cm}\mathtt{def \;CF(}x\mathtt{):} & \hspace{5cm}\\
			\hspace{0.9cm}\mathtt{print }\lfloor x\rfloor &\\
			\hspace{0.9cm}\mathtt{if } \{x\}\neq 0, \; 
			\mathtt{return \; CF}\big(\frac{1}{\{x\}}\big) & \end{array}
	\end{equation*}
	\noindent Note that this algorithm does not always terminate.

	\medskip
	
	The next proposition illustrates why the continued fraction is an extremely interesting algebraic and analytic tool: it unambiguously represents real numbers, provides high-quality rational approximations, and characterizes both rational and quadratic\footnote{Definition: a real number is \emph{quadratic} (resp. \emph{cubic} or \emph{algebraic}) if it is a root of a polynomial of degree $2$ (resp. degree $3$ or any degree) with integer coefficients.} numbers.

	\begin{proposition}\label{prop:reminder_CF}
		Let $x = [a_0;a_1,a_2,a_3,\cdots]$ be a positive real number.
		\begin{enumerate}
			\vspace{-0.2cm}\item[\emph{(1)}] The continued fraction expansion of $x$ is finite if and only if $x \in \Q$.
			\vspace{-0.2cm}\item[\emph{(2)}] \emph{[Dirichlet's theorem]} The sequence of rational numbers
			\[
			\left(\frac{p_n}{q_n}\right)_n := \big([a_0;a_1,a_2,\cdots,a_n]\big)_n = \left(a_0+\frac{1}{a_1+\frac{1}{\ddots+\frac{1}{a_n}}}\right)_n
			\]
			converges efficiently to $x$; moreover, these rational numbers are ``best approximations'' of $x$ in the following sense: one cannot find a rational number that is closer to $x$ than $p_n/q_n$, without paying the price of a larger denominator.
			\vspace{-0.2cm}\item[\emph{(3)}] The following map is a bijection:
			\[
			\begin{array}{ccl}
				\R_{>0}\setminus\Q & \longrightarrow & \N\times(\N_{\geq 1})^\N\\
				x & \longmapsto & (a_n)_{n\in\N}
			\end{array}
			\]
			\vspace{-0.2cm}\item[\emph{(4)}] \emph{[Lagrange's theorem]} An irrational number $x$ is a quadratic if and only if its continued fraction expansion is eventually periodic.
		\end{enumerate}
	\end{proposition}

	It is a long-standing open problem (the question was initially asked by C. Hermite in 1850) to generalize Lagrange's theorem to cubic, or more generally, algebraic numbers. This question is one of the seminal motivations of C. G. J. Jacobi for introducing a ``multidimensional continued fraction'' (on which we shall come back later).

	\subsection{Connection between Sturmian words and continued fractions}\label{subsect:connectionCF_sturmian}
	
	Sturmian words and continued fractions are connected through a renormalization process that we introduce now.
	\medskip
	
	Let $w$ be a Sturmian word. Since $w$ is not eventually periodic, it must contain both factors  $12$ and $21$. Since $p_w(2)=3$, $w$ admits a third factor of length $2$, which can be either $11$ \textbf{xor} (exclusive or) $22$. We say that $w$ is ``of type $1$''  if it contains $11$, and ``of type $2$'' otherwise.
	
	\medskip
	\noindent \emph{The renormalization process. } Let $i \in \{1,2\}$, and denote the other letter by $j$. Let $w$ be a Sturmian word of type $i$. Then, every occurrence of the letter $j$ in $w$ (except the first one when $w$ starts with $j$) is preceded by at least one occurrence of $i$. Let $R_i(w)$ denote the infinite word obtained from $w$ by erasing the $i$ that immediately precedes every $j$. Our renormalization process is then defined by
	\[ R(w) = \begin{cases} R_1(w) \text{ if $w$ is of type 1,} \\ R_2(w) \text{ otherwise.}\end{cases}\]

	\begin{example}
		It is useful to illustrate the renormalization process with an example. To this end, consider again the \emph{Fibonacci word}, which we temporarily assume is Sturmian.
		\[ w_{fibo} = 121121211211212112121121121211211212112121121121211212...\] 
		The Fibonacci word is of type $1$, hence
		\[ R(w_{fibo}) = R_1(w_{fibo}) = 212212122122121221212212212122122...\] 
		We claim that this new infinite word is also Sturmian.
	\end{example}
	
	\begin{theorem}	\label{th:stability_sturmian}
		If $w$ is a Sturmian word, then $R(w)$ is also a Sturmian word.
	\end{theorem}
	
	Therefore, we can iterate $R$ infinitely many times.  We define the ``run-length'' sequence $(b_n)_n$ of a Sturmian word $w$ as the sequence recording the number of times the operation $R_1$ (resp. $R_2$) is repeated before the next application of $R_2$ (resp. $R_1$). Specifically, if  $(R_{i_n})_{n\in \N}$ denotes the successive values taken by the operator $R$ when iterated on $w$, we have
	 \[i_0i_1i_2i_3i_4i_5i_6...=1^{b_0}2^{b_1}1^{b_2}2^{b_3}...\]
	(Note that, by convention, $b_0$ denotes the number of consecutive iterations of $R_1$ before the first application of $R_2$.)

	Since $(R_{i_n})_n$ contains both operators $R_1$ and $R_2$ infinitely many times (see Exercise \ref{exo_UR}), its run-length $(b_n)$ is an \emph{infinite} sequence of \emph{integers}. More precisely, we must have $b_k>0$ for every $k\geq 1$, and $b_0=0$ if and only if the Sturmian word $w$ under consideration is of type $2$.
	
	\begin{example}\label{ex:fibo_derivations} The first iterations of the operator $R$ on the Fibonacci word are:
		\[ \begin{array}{lllll}
			w &= w_{fibo} && = 121121211211212112121121121211211212112...  \quad &\text{type 1}\\
			
			w' &= R(w) &= R_1(w) &= 212212122122121221212212...  & \text{type 2}\\
			
			w'' &= R^2(w) &= R_2(w') &= 121121211211212...  & \text{type 1}\\
			
			w''' & = R^3(w) &= R_1(w'') &= 212212122...  & \text{type 2}\\
			
			\text{etc.} &&&&
		\end{array}\]
		
		In fact, we could prove that $w''=w_{fibo}$ (this is a particularity of the Fibonacci word), and deduce that the run-length of $w_{fibo}$ is the constant sequence equal to $1$.
	\end{example}

	\begin{remark}
		By definition, a ``substitution'' is a function from the alphabet to the set of finite words.  It extends to finite and infinite words as a morphism  with respect to concatenation. (For example, the image of the finite word $aaba$ by the substitution $s: a\mapsto aab$, $b\mapsto ba$ is 
		$s(aaba) = s(a)s(a)s(b)s(a) = aabaabbaaab$.) Now, if $\sigma_1$ and $\sigma_2$ denote the particular substitutions
		\[\begin{array}{lllll}
			\sigma_1 : & 1 \mapsto 1 & \hspace{1cm} & \sigma_2 : & 1 \mapsto 21\\
			& 2 \mapsto 12 &&& 2 \mapsto 2,
		\end{array}\]
		then observe, in Example \ref{ex:fibo_derivations}, that $w$ can be reconstructed from its image $w'$ by applying $\sigma_1$. Similarly, the infinite words $w'$ and $w''$ can be reconstructed from their respective images $w''$ and $w'''$ by $w'=\sigma_2(w'')$ and  $w''=\sigma_1(w''')$.  
		These two substitutions $\sigma_1$ and $\sigma_2$ will play a central role in the sequel.
	\end{remark}
	
	\begin{exercise}\label{exo:resubstitute} Let $w$ be a Sturmian word of type $i$.\\
		1) Prove that the equality $w=\sigma_i(R_i(w))$ does not always hold. (Hint: consider $S(w_{fibo})$.)\\
		2) However, prove that if $u$ is a factor of $R_i(w)$, then $\sigma_i(u)$ is a factor of $w$.
	\end{exercise}

	\begin{proof}[Proof of Theorem \ref{th:stability_sturmian}]
		Let $w$ be a Sturmian word of type $i$, the other letter being again denoted by $j$. We want to prove that $w':=R(w)=R_i(w)$ is also a Sturmian word. First, observe that $w'$ is eventually periodic if and only if $w$ is eventually periodic---which it is not. Therefore, by the Morse--Hedlund theorem \ref{MH38}, we must have $p_{w'}(n)\geq n+1$ for every $n\in\N$.
		
		For the converse inequality, we will prove that, for every $n\in\N$, $w'$ admits at most one right-special factor of length $n$. (Therefore, using again the argument in the proof of Proposition \ref{prop:basics}, Assertion (iv), we have $p(n+1)-p(n)\leq 1$ for all $n\in \N$, whence $p(n)\leq n+1$.) We argue by contradiction.
		Assume that $w'$ contains  two distinct right-special factors of equal length, $u$ and $v$. Then the four extensions $ui$, $uj$, $vi$ and $vj$ are factors of $w'$. By Exercise \ref{exo:resubstitute}, their images under the substitution $\sigma_i$ are factors of the original Sturmian word $w$: $\sigma_i(ui)=\sigma_i(u)i$, $\sigma_i(uj)=\sigma_i(u)ij$, $\sigma_i(vi)=\sigma_i(v)i$, and $\sigma_i(vj)=\sigma_i(v)ij$. We can be slightly more precise by claiming that the two words $\sigma_i(u)ii$ and $\sigma_i(v)ii$ are factors of $w$. Indeed, since $ui$ is a factor of $w'$, either $uii$ or $uij$ (or both) is also a factor of $w'$, from which it follows that either $\sigma_1(u)ii$ or $\sigma_1(u)iij$ is a factor of $w$; in either case, $\sigma_1(u)ii$ is a factor of $w$---the argument being analogous for   $\sigma_1(v)ii$.
		
		At this point, we have established that the two finite words $\sigma_i(u)i$ and $\sigma_i(v)i$, which are distinct since $u\neq v$, form two right-special factors of the Sturmian word $w$. This is not yet a contradiction, since these two factors may not have the same length. However, since $u$ and $v$ are equally-long distinct words, they can always be factorized as (up to exchanging $u$ and $v$) 
		\[\begin{cases} u = pis, \\
			v = rjs, \end{cases}\]
		where $p, r$ and $s$ are possibly empty finite words. Therefore, the suffixes of their images by $\sigma_i$, extended by the letter $i$, namely $i\sigma_i(s)i$ and $j\sigma_i(s)i$ are two distinct, equally-long, right-special factors of the Sturmian word $w$. This contradicts Proposition \ref{prop:basics}.
	\end{proof}

	\begin{exercise}\label{exo_UR} Let $w$ be a Sturmian word.\\
		1) Prove that both operators $R_1$ and $R_2$ occur infinitely many times in the successive renormalizations of $w$.\\
		2) Prove that $w$ is \emph{uniformly recurrent}. (Definition: $w$ is ``uniformly recurrent'' if for every factor $u$ of $w$, there exists an integer $N$ such that any two successive occurrences of $u$ in $w$ are separated by at most $N$ letters.)
	\end{exercise}

	The next theorem establishes the connection between Sturmian words, our renormalization process, and continued fractions.

	\begin{theorem}\label{th:connection_CF}
		Let $w$ be a Sturmian word. Denote by $(b_n)_{n\in\N}$ its run-length  under  iteration of the renormalization process $R$. Then
		\begin{enumerate}
			\vspace{-0.2cm}\item[\emph{(i)}]  $w$ admits letter frequencies (hereafter denoted by $f_1,f_2$);
			\vspace{-0.2cm}\item[\emph{(ii)}]  $[b_0;b_1,b_2,b_3,\cdots]$ is the continued fraction expansion of the ratio $\frac{f_1}{f_2}$;
			\vspace{-0.2cm}\item[\emph{(iii)}] the ratio $\frac{f_1}{f_2}$ is an irrational number.
		\end{enumerate}
	\end{theorem}

	By definition, the ``frequency of the letter $a$'' in an infinite word $w$ is the limit, if it exists, of the proportion of $a$ in growing prefixes of $w$. Specifically, it is the limit (again, if it exists)
	\[     f_a = \underset{n\to+\infty}{\lim}\,\frac{|\pref_n(w)|_a}{n},
	\]
	where $|\pref_n(w)|_a$ denotes the number of $a$ among the first $n$ letters of $w$.
	
	\begin{example}\label{ex:fibonacci_frequencies}
		Since the run-length of the Fibonacci word is the infinite sequence constant equal to $1$, the quotient of its letter frequencies $f_1/f_2$ is the golden ratio.
	\end{example}

	\begin{proof}[Proof of Theorem \ref{th:connection_CF}, assertion (i)] 
		Let $w$ be a Sturmian word.  Since $w$ is uniformly recurrent (as proven in Exercise \ref{exo_UR}) and satisfies
		\[
		\underset{n\to+\infty}{\liminf}\,\frac{p_w(n)}{n}=1,
		\]
		it follows from a theorem by Boshernitzan from 1984 that $w$ admits letter frequencies \cite{Bos84} (see also \cite[Section~7.3]{FM10}). \end{proof}
	
	\begin{remark}
		Classically, the existence of letter frequencies in Sturmian words is derived from an abelian property called ``balancedness'', which we choose not to develop in this lecture. Boshernitzan's theorem enables us to establish the same result, directly from the complexity.
	\end{remark}

	\begin{proof}[Proof of Theorem \ref{th:connection_CF}, assertions (ii) and (iii)] Let $w$ be a Sturmian word. Denote by $f_1$ and $f_2$ its letter frequencies. \\ (ii)  By definition of the run-length sequence $(b_n)_n$ of $w$, the first term $b_0$ is the number of times the operator $R$ initially acts as $R_1$ before switching to $R_2$. Step 1: initial iterations of $R_1$. Assume that $b_0\neq0$ (otherwise, we bypass this step and work directly with $b_1$). Let $w'=R(w)=R_1(w)$. Since $w$ and $w'$ are Sturmian words (Theorem \ref{th:stability_sturmian}), they admit letter frequencies, respectively denoted by $f_1$, $f_2$, $f'_1$ and $f'_2$. We claim that these letter frequencies satisfy the relation
		\begin{equation}\label{eq:ff'}
			\frac{f'_1}{f'_2} = \frac{f_1}{f_2} -1. 
		\end{equation} Indeed, it suffices to count the number of letters in the image of the prefixes of length $n$, $n \in \N$,  by the operator $R_1$:
		\[ \begin{cases} \vert R_1(\pref_n(w)) \vert_2 = \vert \pref_n(w) \vert_2 
			\\\vert R_1(\pref_n(w)) \vert_1 = \vert \pref_n(w) \vert_1 - \vert \pref_n(w) \vert_2  \text{ or }  \vert \pref_n(w) \vert_1 - \vert \pref_n(w) \vert_2 + 1
		\end{cases}\]
		(The number of occurrences of the letter $2$ is unchanged; we remove one $1$ before each occurrence of the letter $2$, except the possible initial $2$.) Since $(R_i(\pref_n(w)))_n$ is a sequence of growing prefixes of $w'$, we have
		\[ \frac{\vert R_1(\pref(w)) \vert_1}{\vert R_1(\pref(w)) \vert} \longrightarrow_{n\rightarrow \infty} f'_1 
		\]
		and a symmetrical expression for $f'_2$ (we recall that $\vert u \vert$ denotes the length of a finite word $u$). Therefore
		\[\frac{f'_1}{f'_2} = \lim_n \frac{\vert R_1(\pref(w)) \vert_1}{\vert R_1(\pref(w)) \vert_2} = \lim_n \frac{\vert \pref(w) \vert_1 - \vert \pref(w) \vert_2}{\vert \pref(w) \vert_2} = \frac{f_1}{f_2} -1 \]
		which establishes \eqref{eq:ff'}.
		Now, the operation $R_1$ being repeated exactly $b_0$ times before switching to $R_2$, the letter frequencies of $w'':=R^{b_0}(w) = (R_1)^{b_0}(w)$ must  satisfy
		\begin{equation}\label{eq:f''}
			\begin{cases}
				f''_1/f''_2 = f_1/f_2 - b_0 \\
				f''1/f''2 < 1
			\end{cases}
		\end{equation}
		(In the second expression, the large inequality derives from the fact that $w''$ is of type $2$, while the equality case is forbidden by Exercise \ref{exo:not_one_half} below.)
		The expression \eqref{eq:f''} is equivalent to
		\[
		\begin{cases}
			b_0 = \lfloor f_1/f_2 \rfloor\\
			f''_1/f''_2 = \{ f_1/f_2\} 
		\end{cases}
		\]
		(which already resembles the continued fraction...)
		\bigskip
		
		\noindent Step 2: iterating $R_2$. We now apply the operator $R_2$ to the new Sturmian word $w''$,  exactly $b_1$ times. We proceed as in Step 1, except that the roles of the letters $1$ and $2$ are exchanged since $w''$ is of type $2$. Thus, instead of considering the quotient $\frac{f_1}{f_2}$, we consider the quotient
		\[
		\frac{f''_2}{f''_1}=\frac{1}{\big\{\frac{f_1}{f_2}\big\}}.
		\]
		It is now clear that we are  performing the algorithm \texttt{CF} (written in Section \ref{subsect:crash_courseCF}) on the input $x=f_1/f_2$, which displays the continued fraction of $x$.

		\medskip
		\noindent (iii) Since $w$ is Sturmian, its sequence of renormalizations $(R_n)_n$ contains infinitely many occurrences of both operators $R_1$ and $R_2$ (Exercise \ref{exo_UR}). This implies that its run-length $(b_n)$---which we proved is the continued fraction expansion of $f_1/f_2$---is infinite. It follows from the classical Proposition \ref{prop:reminder_CF} that the ratio $f_1/f_2$ is an irrational number.
	\end{proof}

	\begin{exercise}\label{exo:not_one_half}  Prove, without using Theorem \ref{th:connection_CF} assertions (ii) and (iii) (otherwise the theorem's proof is circular), that no Sturmian word has letter frequencies $f_1=f_2=1/2$. \\
		\emph{Hint: if $w$ is a Sturmian word of type $i$, partition $w$ into consecutive blocks of length $N$, where the integer $N$ is chosen such that all factors of length $N$ contain at least two occurrences of the word $ii$. You will need the property of uniform recurrence}. 
	\end{exercise}
	
	
	\subsection{An efficient construction of Sturmian words}

	In this subsection, we prove that Sturmian words exist, and that they are uncountably many: essentially one per positive irrational number. These results rely on an efficient construction of Sturmian words with the continued fraction.
	
	\medskip
	
	The next theorem presents the construction. It uses the substitutions $\sigma_1$ and $\sigma_2$, which appeared in Section \ref{subsect:connectionCF_sturmian} as reverses of the renormalization process $R$. They were defined by
	\[\begin{array}{lllll}
		\sigma_1 : & 1 \mapsto 1 & \hspace{1cm} & \sigma_2 : & 1 \mapsto 21\\
		& 2 \mapsto 12 &&& 2 \mapsto 2.
	\end{array}\] 
	
	\begin{theorem}\label{th:standard_sturmian}
		Let $(s_n)_n\in\{\sigma_1,\sigma_2\}^\N$ be a sequence in which both substitutions $\sigma_1$ and $\sigma_2$ occur infinitely often. Then\\
		$(i)$ the two sequences of finite words $(s_1\circ s_2\circ\cdots\circ s_n(1))_n$ and $(s_1\circ s_2\circ\cdots\circ s_n(2))_n$ converge, as $n\to\infty$, to the same infinite word $w_0$, which is Sturmian;\\
		$(ii)$ the run-length of the substitution sequence $(s_n)_n$ (defined, again, as the number of times $\sigma_1$ is repeated before the first apparition of $\sigma_2$, then the number of times $\sigma_2$ is repeated before the next apparition of $\sigma_1$, and so on)  is exactly the continued fraction expansion of the ratio $\frac{f_1}{f_2}$, where $f_1$ and $f_2$ denote the letter frequencies in the Sturmian word $w_0$.
	\end{theorem}
	
	\noindent (The topology we use is the Tychonoff topology for the product set $\{1,2\}^{\N}$: a sequence of finite or infinite words $(u_n)_n$ converges to an infinite words $w$ if the words $u_n$ share larger and larger prefixes with $w$.)

	\begin{example} \label{ex:def_fibo} Let $(s_n)_n=\sigma_1,\sigma_2,\sigma_1,\sigma_2,\sigma_1,\sigma_2,\ldots$ be the periodic sequence with period $(\sigma_1,\sigma_2)$. Then, the sequence of finite words $u_n := (s_1\circ s_2\circ\cdots\circ s_n(1))_n$, which starts with 
		\[    \begin{array}{l}
			u_0 = u_1 = 1 \\
			u_2 = u_3 = 121 \\
			u_4 = u_5 = 12112 \\
			u_5 = u_6 = 1211212112112 \\
			u_7 = u_8 = 1211212112112121121211211212112112 \\
			...  \end{array}
		\]
		converges to a Sturmian word $w_0$. This Sturmian word $w_0$ is called ``Fibonacci word''--- it is the example of a Sturmian word that we mentioned in advance in this lecture. As we announced in Examples \ref{ex:fibo_derivations} and \ref{ex:fibonacci_frequencies}, the ratio of its letter frequencies is the golden ratio. Indeed, the run length of the substitution sequence $(s_n)_n$ is the constant sequence, equal to $1$; therefore, $f_1/f_2 = [1;1,1,1,1,1,1,1,...] = \varphi$. The following exercise justifies its name ``Fibonacci''.
	\end{example}
	
	\begin{exercise}
		Let $(v_n)_n$ be the sequence of finite words defined by 
		\[\begin{cases}
			v_0 = 2\\ v_1 = 1\\  v_{n+2} = v_{n+1}+v_n \text{ for every } n\geq 0,
		\end{cases}\] where the symbol $+$ denotes the concatenation operation. Prove that $(v_n)_n$ converges to the Fibonacci word.
	\end{exercise}

	\begin{definition} The Sturmian words obtained in Theorem \ref{th:standard_sturmian} are called \emph{standard Sturmian words}. The sequence of substitutions $(s_n)_n \in \{\sigma_1,\sigma_2\}^{\N}$ from which a standard Sturmian word is constructed is called its \emph{directive sequence}. 
	\end{definition}

	A natural question is: is every Sturmian word standard, that is, is every Sturmian word the limit of a sequence of iterated images by the substitutions $\sigma_1$ and $\sigma_2$? The answer is `no'. However,  every Sturmian word is  related to a standard Sturmian word.
	
	\begin{exercise}\label{exo:subshif}
		1) Prove that there exist non-standard Sturmian words.\\
		2) However, prove that for every Sturmian word $w$, there exists a unique standard Sturmian word $w_0$ such that $w$ belongs to the ``subshift'' of $w_0$
		\[
		w\in\overline{\{S^n(w_0)\,|\,n\in\N\}},
		\]
		or equivalently, such that $w$ and $w_0$ admit the same set of factors.
	\end{exercise}
	
	The next result is a corollary of Theorem \ref{th:standard_sturmian}. It states that there exist uncountably many Sturmian words, essentially one per positive irrational number.
	
	\begin{corollary}
		For every positive irrational number $x$, there exists a unique standard Sturmian word $w_0$ whose letter frequencies satisfy $\frac{f_1}{f_2}=x$.
	\end{corollary}
	
	\begin{proof}
		By Proposition \ref{prop:reminder_CF}, the continued fraction is a bijection between the set of positive irrational numbers and the set of infinite sequences of integers $\N\times(\N_{\geq 1})^{\N}$, which is itself in bijection, through the notion of ``run length'', with the set of sequences of substitutions $\{\sigma_1,\sigma_2\}^{\N}$ in which $\sigma_1$ and $\sigma_2$ appear infinitely many times. 
		Therefore, by Theorem \ref{th:standard_sturmian}, for every positive irrational number $x$, there exists a unique standard Sturmian word whose letter frequencies satisfy $f_1/f_2 =x$.
	\end{proof}
	
	We conclude this subsection with a remark and the proof of Theorem \ref{th:standard_sturmian}.
	
	\begin{remark}  The construction of Sturmian words  presented in this section is implemented in the general mathematical software system \texttt{Sagemath}. To illustrate the efficiency of the algorithm, let us say that the knowledge of the first $n$ terms of the continued fraction of $f_1/f_2$ determines the first $\simeq \exp(n)$ letters of the Sturmian word $w_0$. \end{remark}
	
	\begin{proof}[Proof of Theorem \ref{th:standard_sturmian}]
		
		Let $(s_n)_n\in\{\sigma_1,\sigma_2\}^\N$  be a sequence that contains each substitution $\sigma_1$ and $\sigma_2$ infinitely many times. For $n\in \N$, we denote $u_n:=s_1\circ\cdots\circ s_n(1)$ and $v_n:=s_1\circ\cdots\circ s_n(2)$. We recall that $|u|$ denotes the length of a finite word $u$. We are going to prove that:
		\begin{enumerate}
			\item[1)] ``the sequences of integers $(|u_n|)_n$ and $(|v_n|)_n$ are non-decreasing, and increase for infinitely many $n$.'' As a consequence, if $(u_n)_n$ and $(v_n)_n$ converge, they must converge to infinite words.
			\item[2)] ``the sequences $(u_n)_n$ and $(v_n)_n$ converge to  a common limit.'' This limit will be denoted by $w_0$.
			\item[3)] ``the infinite word $w_0$ admits exactly one right-special factor of length $n$, for every $n\in\N$.'' Therefore, its complexity is $n\mapsto n+1$, and $w_0$ is a Sturmian word.
			\item[4)] ``the sequence of successive renormalizations $(Q_n)_n\in\{R_1,R_2\}^\N$ of the Sturmian word $w_0$ is given by
			\[
			Q_n=R_1 \iff s_n=\sigma_1.''
			\]
			Therefore, the run lengths of $(Q_n)_n$ and $(s_n)_n$ coincide, and are equal (by Theorem~\ref{th:connection_CF}) to the continued fraction expansion of $f_1/f_2$, the ratio of the letter frequencies of $w_0$. At the end of this step, the theorem will be completely proven.
		\end{enumerate}

		1) Since the image of any letter by $\sigma_1$ and $\sigma_2$ contains at least one letter, we always have $|u_{n+1}|\geq |u_n|$. Furthermore, since the sequence $(s_n)_n$ contains infinitely many occurrences of $\sigma_1$ and $\sigma_2$, for infinitely many $n\in\N$, we have $u_{n+2}=\sigma_1\circ\sigma_2(u_n)$. It follows from $\sigma_1\circ\sigma_2(1)=121$ and $\sigma_1\circ\sigma_2(2)=12$ that $|u_{n+2}|>|u_n|$ for these integers $n$. Therefore, the sequence $(|u_n|)_n$ is non-decreasing, and increases infinitely many times. A similar argument holds for the sequence $(v_n)_n$.

		\bigskip
		2) Let $n\in\N$. Observe that
		
		- if $s_{n+1}=\sigma_1$, then
		\[
		\begin{array}{l}
			u_{n+1}=s_1\circ\cdots\circ s_{n+1}(1)=s_1\circ\cdots\circ s_n(1)=u_n,\\
			v_{n+1}=s_1\circ\cdots\circ s_{n+1}(2)=s_1\circ\cdots\circ s_n(12)=u_nv_n.
		\end{array}
		\]
		
		- similarly, if $s_{n+1}=\sigma_2$, then $u_{n+1}=v_n u_n$ and $v_{n+1}=v_n$.
		
		\noindent
		Thus, if for every $n\in\N$ we set
		\[
		p_{n+1}:=\begin{cases}
			u_n & \text{ if $s_{n+1}=\sigma_1$,} \\
			v_n & \text{ otherwise,}
		\end{cases}
		\]
		then $p_{n+1}$ is a common prefix of $u_{n+1}$ and $v_{n+1}$, and hence of $p_{n+2}$. As a consequence, since $|p_n|\to\infty$, the three sequences of finite words $(u_n)_n$, $(v_n)_n$ and $(p_n)_n$ converge to a common infinite word $w_0\in\{1,2\}^\N$.
		
		\bigskip
		3) We start with a preliminary observation: the image of a finite or infinite word by $\sigma_1$ does not contain the factor $22$, and symmetrically, the image of a finite or infinite word by $\sigma_2$ does not contain the factor $11$.
		
		\smallskip
		First, we prove that for every $n\in\N$, $w_0$ admits at most one right-special factor of length $n$. We proceed by contradiction. Let $u\neq v$ be two right-special factors of equal length in $w_0$. Since $u_n$ converges to $w_0$, there exists $n$ such that $u1$, $u2$, $v1$ and $v2$ appear in $u_n = s_1\circ\cdots\circ s_n(1)$. We are going to ``desubstitute'' $u1$, $u2$, $v1$ and $v2$ by $s_1$, and obtain two shorter, equally long right-special factors $u''$ and $v''$  of $s_2\circ\cdots\circ s_n(1)$.
		
		To proceed with rigor, we introduce the following definition:  $u'$ is a \emph{pre-image} of a finite word $u$ by a substitution $s$ if $u$ is a factor of $s(u')$, and $u'$ is minimal for this property (the order being: $u<v$ if and only if $u$ is a factor of $v$). For example, $212$ is the unique pre-image of $12112$ by $\sigma_1$, while $211$ and $212$ are two pre-images of $1211$.
		
		We resume the proof. Without loss of generality, assume that $s_1 = \sigma_1$. In this case, the last letter of the right-special factors $u$ and $v$ must be $1$ (otherwise, the word $22$ would appear in $w_0$). Then, when they are followed by the letter $2$ (as in $u2$ and $v2$), these last $1$ necessarily originate from a letter $2$ in $s_2\circ\cdots\circ s_n(1)$. By contrast, when they are followed by another $1$ (as in $u1$ and $v1$), these last $1$ must originate from a letter $1$ in $s_2\circ\cdots\circ s_n(1)$. Now, reading $u$ and $v$ backwards from their penultimate letters, we see that the rest of their pre-image by $\sigma_1$ is uniquely determined.  Let us call them $u'$ and $v'$. The words $u'$ and $v'$ are thus two distinct right-special factors of $s_2\circ\cdots\circ s_n(1)$. Furthermore, their lengths satisfy $|u'|, |v'| < |u| = |v|$, and they cannot be prefixes of each other (otherwise, $|u|\neq|v|$). Therefore, by cropping a few letters  at the beginning of one of them if necessary, we can ensure that $u'$ and $v'$ are two distinct, equally long right-special factors of  $s_2\circ\cdots\circ s_n(1)$, with $|u'|=|v'| < |u|=|v|$. Now, by iterating this ``desubstitution'' finitely many times, we will reach an integer $k$ and $u''\neq v''$ two distinct right-special factors of length $1$, in $s_k\circ\cdots\circ s_n(1)$. Such a situation is impossible:  our preliminary observation guarantees that $11$ and $22$ cannot simultaneously appear in $s_k\circ\cdots\circ s_n(1)$. A contradiction. 
		
		\smallskip
		We now want to prove that for every $n\in\N$, $w_0$ admits at least one right-special factor of length $n$. Since every suffix of a right-special factor is also a right-special factor,  it suffices to prove the property for infinitely many $n$. Observe that both factors $11$ and $12$ appear in all finite words of the form $\sigma_1\circ\sigma_2^k\circ\sigma_1(2)$, with $k\geq 1$.
		Since the patterns $\sigma_1\circ \sigma_2^k\circ\sigma_1$ have to appear infinitely many times in $(s_n)_n$, we deduce that for infinitely many $n$, the finite words $s_1\circ\cdots\circ s_n(11)$ and $s_1\circ\cdots\circ s_n(12)$ are factors of $w_0$. These two words are distinct. Therefore, their longest common prefix (which  starts with $u_n=s_1\circ\cdots\circ s_n(1)$) is a right-special factor of length larger than $|u_n|$. By the first step, we conclude that $w_0$ admits arbitrarily long right-special factors. This completes the proof of the third step: $w_0$ admits exactly one right-special factor for every length $n\in \N$. 
		
		\bigskip
		4) It suffices to notice that if $s_1=\sigma_1$, then $w_0 = \lim_{k\rightarrow \infty} s_1 \circ \cdots \circ s_k(1)$ is of type $1$, and $R(w_0)=R_1(w_0)$ is exactly the Sturmian word $w_1 := \lim_{k\rightarrow \infty} s_2 \circ ... \circ s_k(1)$. Symmetrically, if $s_1 = \sigma_2$, then $w_0$ is of type $2$ and $R(w_0)=R_2(w_0)=\lim_{k\rightarrow \infty} s_2 \circ ... \circ s_k(1)$. We conclude by induction.
	\end{proof}

	%
	%

	\subsection{Equivalent dynamical and geometrical descriptions of Sturmian words}


	\definecolor{rouge}{rgb}{1,0,0}
	\definecolor{gris}{rgb}{0.3764705882352941,0.3764705882352941,0.3764705882352941}
	\definecolor{vert}{rgb}{0,0.8,0}
	\definecolor{orange}{rgb}{1,0.6,0.2}
	\definecolor{bleu}{rgb}{0,0.2,0.8}
	\definecolor{gris2}{rgb}{0.6274509803921569,0.6274509803921569,0.6274509803921569}
	
	In this last subsection, we present (with no proof), alternative dynamical and geometrical constructions of Sturmian words. Specifically, we will see that
	Sturmian words encode the \emph{minimal} trajectories of
	
	- a billiard ball on a square table,
	
	- a point subjected to the linear flow on the two-dimensional torus,
	
	- a rotation on the circle,
	
	\noindent and that they ``digitize'' the half-lines of $\R^2$ with irrational slopes.
	
	\smallskip
	\noindent (We recall that a topological dynamical system is \emph{minimal} when every point has a dense trajectory. In the next chapter, we will discuss the importance of the minimal and irrational assumptions.)
	\bigskip
	
	We first define the dynamical and geometrical systems mentioned above, and explain how to encode them with infinite words.
	
	\paragraph{Billiard words.} Let $x \in [0,1)^2$ and $\theta \in (\R_{>0})^2$. The billiard word $w=w_{\mathrm{billiard}}(x,\theta)$ is the infinite binary word encoding the trajectory of a billiard ball, initially located at position $x$, and launched with a momentum $\theta$, on the square table $[0,1]^2$. Its $n$-th letter  is defined by:
	\[w[n] = \begin{cases} 1  \text{ if the $n$-th edge hit by the ball is vertical,}\\  \text{2 if the $n$-th edge hit by the ball is horizontal.} \end{cases}\]
	(Since it will not be useful in this lecture, we choose not to encode trajectories that hit corners.)
	Example: In Figure \ref{Fig:dynamicals_sturmian}, the billiard word starts with $w_{\mathrm{billiard}}(x,\theta)=12112$...
	
	\paragraph{Digitization of straight lines.} Again, let $x \in [0,1)^2$ and $\theta \in (\R_{>0})^2$. Denote by $\mathcal{G} = \Z\times\R \cup \R\times\Z$ the integer grid of $\R^2$. The \emph{cutting sequence}  $w=w_{line}(x,\theta)$ is the infinite binary word that digitizes the half-line $\mathcal{D}$ parameterized by $x+t\theta$, $t \in \R_{>0}$. Its $n$-th letter is defined by
	\[w[n] = \begin{cases} 1  \text{ if the $n$-th  intersection of the line $\mathcal{D}$ with the   grid $\mathcal{G}$ is on a vertical branch,}\\  \text{2 otherwise.} \end{cases}\]
	(We do not digitize the lines that pass through a grid intersection.)
	Example: The half-line of Figure 2\ref{Fig:dynamicals_sturmian} is digitized by an infinite word which starts with $w_{line}(x,\theta)=12112...$. Interpretation: if we interpret a letter $1$ as a unit step to the right, and the letter $2$ as a unit step upwards, the word $w$ is the ``staircase'' that best represents the line $\mathcal{D}$.
	
	\paragraph{Coding a linear flow on the torus. } Let $x \in [0,1)^2$ and $\theta \in (\R_{>0})^2$. The ``flow word"  $w=w_{flow}(x,\theta)$ is the infinite binary word that encodes the trajectory of the point $x$ under the action of the linear flow with direction $\theta$ on the torus $\R^2/\Z^2 = [0,1)^2$. If $P$ denotes the Poincar\'e section $P= \{0\}\times[0,1)\cup [0,1)\times\{0\}$, the coding is defined by
	\[ w[n] = \begin{cases} 1 \text{ if the $n$-th time $P$ is crossed, it is crossed via its vertical branch,} \\ 2 \text{ otherwise.} \end{cases} \]
	(We do not encode the trajectories passing through the corner point $(0,0)$ of the torus.)
	Example: The flow-trajectory of $x$ in Figure \ref{Fig:dynamicals_sturmian} is encoded by an infinite word $w_{flow}(x,\theta) = 12112...$

	\paragraph{Coding of rotations. } The rotation by angle $\alpha \in \R$ on the circle $\R/\Z = [0,1)$ is the map $R_{\alpha}: y \mapsto y + \alpha \mod 1.$
	The trajectory of a point $y \in [0,1)$ is encoded by the infinite binary word $w = w_{rotation}(y,\alpha)$ whose $n$-th letter is defined by 
	\[ w[n] = \begin{cases} 1 \text{ if } \; R_{\alpha}^{n-1} (y) \in [0,1-\alpha) \\ 2 \text{ otherwise.}\end{cases}\]
	Example: In Figure \ref{Fig:dynamicals_sturmian}, the trajectory of $y$ is encoded by the infinite word $w_{rotation}(y,\alpha)=12112...$.
	
	
	\definecolor{rouge}{rgb}{1,0,0}
	\definecolor{gris}{rgb}{0.3764705882352941,0.3764705882352941,0.3764705882352941}
	\definecolor{vert}{rgb}{0,0.8,0}
	\definecolor{orange}{RGB}{253,190,133} 
	\definecolor{orange2}{RGB}{253,174,107} 
	\definecolor{bleu}{RGB}{31,119,180} 
	\definecolor{gris2}{rgb}{0.6274509803921569,0.6274509803921569,0.6274509803921569}
	
	
	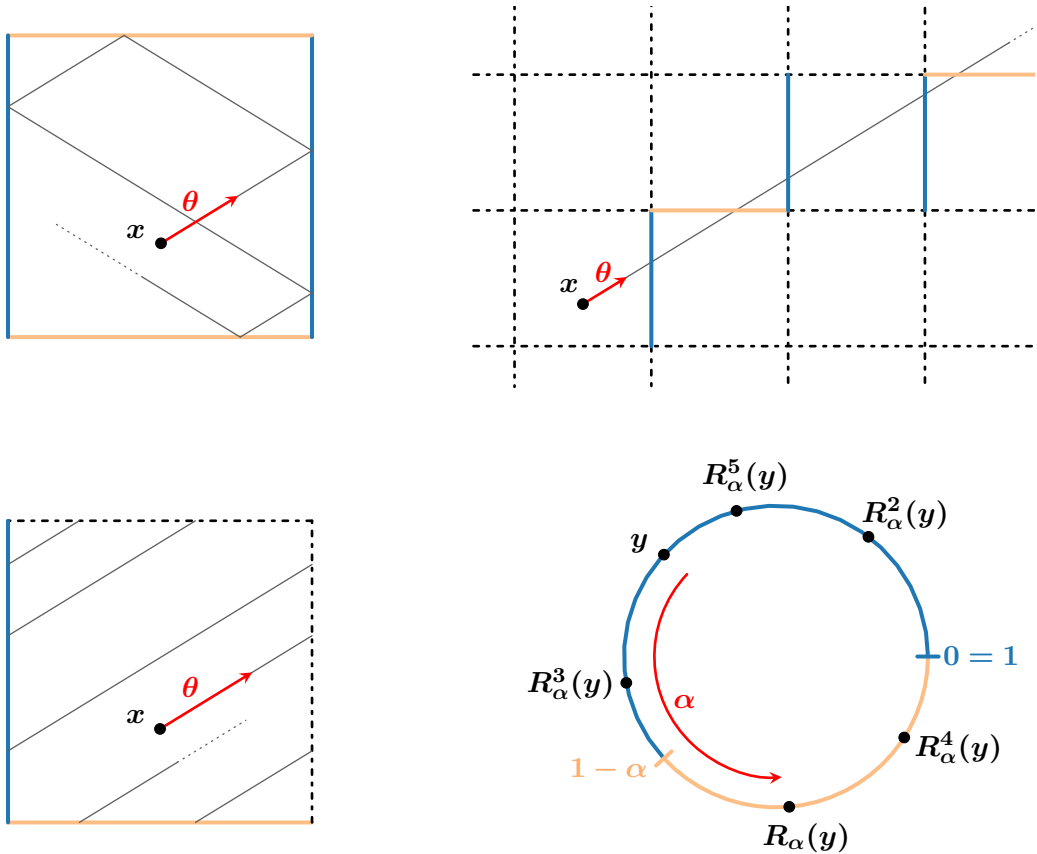
\begin{figure}
		\begin{center}
			\begin{tikzpicture}[line cap=round,line join=round,>=stealth,x=4cm,y=4cm]
				\clip(-.1,-.3) rectangle (1.3,1.1);
				\draw [line width=1.5pt, color=orange] (0,0)-- (1,0);
				\draw [line width=1.5pt, color=bleu] (1,0)-- (1,1);
				\draw [line width=1.5pt, color=orange] (1,1)-- (0,1);
				\draw [line width=1.5pt, color=bleu] (0,1)-- (0,0);
				
				\draw [line width=0.5pt,color=gris] (0.5032443779377416,0.311022130212822)-- (1,0.6180339887498948);
				\draw [line width=0.5pt,color=gris] (1,0.6180339887498948)-- (0.38196601125010504,1);
				\draw [line width=0.5pt,color=gris] (0.38196601125010504,1)-- (0,0.7639320225002106);
				\draw [line width=0.5pt,color=gris] (0,0.7639320225002106)-- (1,0.14589803375031593);
				\draw [line width=0.5pt,color=gris] (1,0.14589803375031593)-- (0.7639320225002091,0);
				\draw [line width=0.5pt,color=gris] (0.7639320225002091,0)-- (0.4433818569265232,0.19811089742394455);
				\draw [line width=0.5pt,dotted,color=gris] (0.4433818569265232,0.19811089742394455)-- (0.1474147733799961,0.3810286146068782);
				
				\draw [->,line width=1pt,color=rouge] (0.5032443779377416,0.311022130212822) -- (0.7572681306014053,0.46801744330876277);
				\draw [fill] (0.5032443779377416,0.311022130212822) circle (2pt);
				\draw (0.42,0.35) node {$\bm{x}$};
				\draw (0.6,0.45) node {\textcolor{red}{$\bm{\theta}$}};
			\end{tikzpicture}
			~
			\begin{tikzpicture}[line cap=round,line join=round,>=stealth,x=1.8cm,y=1.8cm]
				\clip(-0.6,-0.6) rectangle (3.8,2.5);
				
				\draw [line width=0.5pt, color=gris] (0.5,0.31)-- (3.6114185206287814,2.2329623989745024);
				\draw [line width=0.5pt,dotted, color=gris] (3.6114185206287814,2.2329623989745024)-- (3.810072968553901,2.3557375998085726);
				
				\draw [line width=1pt,dash pattern=on 2pt off 3pt on 1pt off 3pt] (0,3.7523525335395043)--(0,-0.3);
				\draw [line width=1pt,dash pattern=on 2pt off 3pt on 1pt off 3pt] (1,3.7754241683027647)--(1,1); \draw [line width=1pt,dash pattern=on 2pt off 3pt on 1pt off 3pt] (1,0)--(1,-0.3);
				\draw [line width=1pt,dash pattern=on 2pt off 3pt on 1pt off 3pt] (2,3.7408167161578736)--(2,2); \draw [line width=1pt,dash pattern=on 2pt off 3pt on 1pt off 3pt] (2,1)--(2,-0.3);
				\draw [line width=1pt,dash pattern=on 2pt off 3pt on 1pt off 3pt] (3,3.8965502508098835)--(3,2); \draw [line width=1pt,dash pattern=on 2pt off 3pt on 1pt off 3pt] (3,1)--(3,-0.3);
				\draw [line width=1pt,dash pattern=on 2pt off 3pt on 1pt off 3pt] (4,3.827335346520101)--(4,3); \draw [line width=1pt,dash pattern=on 2pt off 3pt on 1pt off 3pt] (4,1)--(4,-0.3);
				\draw [line width=1pt,dash pattern=on 2pt off 3pt on 1pt off 3pt] (-0.3,3)--(3,3); \draw [line width=1pt,dash pattern=on 2pt off 3pt on 1pt off 3pt] (4,3)--(4.70361611220184,3);
				\draw [line width=1pt,dash pattern=on 2pt off 3pt on 1pt off 3pt] (-0.3,2)--(3,2); \draw [line width=1pt,dash pattern=on 2pt off 3pt on 1pt off 3pt] (4,2)--(4.790134742564068,2);
				\draw [line width=1pt,dash pattern=on 2pt off 3pt on 1pt off 3pt] (-0.3,1)--(1,1); \draw [line width=1pt,dash pattern=on 2pt off 3pt on 1pt off 3pt] (2,1)--(4.795902651254883,1);
				\draw [line width=1pt,dash pattern=on 2pt off 3pt on 1pt off 3pt] (-0.3,0)--(4.790134742564068,0);
				
				\draw [line width=1.5pt, color=bleu] (1,1)-- (1,0);
				\draw [line width=1.5pt, color=orange] (1,1)-- (2,1);
				\draw [line width=1.5pt, color=bleu] (2,1)-- (2,2);
				\draw [line width=1.5pt, color=bleu] (3,2)-- (3,1);
				\draw [line width=1.5pt, color=orange] (3,2)-- (4,2);
				\draw [line width=1.5pt, color=bleu] (4,2)-- (4,1);
				
				\draw [->,line width=1pt,color=rouge] (0.5,0.31) -- (0.8216762615878631,0.5088068630353015);
				\draw [fill] (0.5,0.31) circle (2pt);
				\draw (0.4,0.45) node {$\bm{x}$};
				\draw (0.65,0.55) node {\textcolor{red}{$\bm{\theta}$}};
			\end{tikzpicture}
			\\
			\begin{tikzpicture}[line cap=round,line join=round,>=stealth,x=4cm,y=4cm]
				\clip(-0.1,-0.1) rectangle (1.48,1.3);
				
				\draw [line width=1.5pt, color=orange] (0,0)-- (1,0);
				\draw [line width=1pt, dash pattern=on 2pt off 3pt on 1pt off 3pt] (1,0)-- (1,1);
				\draw [line width=1pt, dash pattern=on 2pt off 3pt on 1pt off 3pt] (1,1)-- (0,1);
				\draw [line width=1.5pt, color=bleu] (0,1)-- (0,0);
				
				\draw [line width=0.5pt, color=gris] (0.5,0.31)-- (1,0.6190169943749475);
				\draw [line width=0.5pt, color=gris] (0,0.6190169943749475)-- (0.6164434522374274,1);
				\draw [line width=0.5pt, color=gris] (0.6164434522374274,0)-- (1,0.23705098312484235);
				\draw [line width=0.5pt, color=gris] (0,0.23705098312484235)-- (1,0.8550849718747372);
				\draw [line width=0.5pt, color=gris] (0,0.8550849718747372)-- (0.23447744098732223,1);
				\draw [line width=0.5pt, color=gris] (0.23447744098732223,0)-- (0.5571615021721624,0.1994297174400819);
				\draw [line width=0.5pt,dotted, color=gris] (0.5571615021721624,0.1994297174400819)-- (0.7824576346081474,0.33867038481941825);
				
				\draw [->,line width=1pt,color=rouge] (0.5,0.31) -- (0.8031851677125179,0.4973787385311733);
				\draw [fill] (0.5,0.31) circle (2pt);
				\draw (0.42,0.35) node {$\bm{x}$};
				\draw (0.6,0.45) node {\textcolor{red}{$\bm{\theta}$}};
			\end{tikzpicture}
			~
			\begin{tikzpicture}[line cap=round,line join=round,>=stealth, x=2cm,y=2cm]
				\clip(-1.9,-1.3) rectangle (1.7,1.35);
				\draw [shift={(0,0)},line width=1.5pt,color=bleu]  plot[domain=0:3.8832220774509327,variable=\t]({cos(\t r)},{sin(\t r)});
				\draw [shift={(0,0)},line width=1.5pt,color=orange]  plot[domain=-2.3999632297286535:0,variable=\t]({cos(\t r)},{sin(\t r)});
				\draw [line width=1.5pt, color=bleu] (0.93,0)-- (1.07,0);
				\draw[color=bleu] (1.35,0.01) node {$\bm{0=1}$};
				\draw [line width=1.5pt, color=orange] (-0.789,-0.7228)-- (-0.6858,-0.6282);
				\draw[color=orange2] (-1.1,-0.75) node {$\bm{1-\alpha}$};
				\draw [fill] (-0.7373688780783197,0.6754902942615238) circle (2pt);
				\draw (-0.9,0.75) node {$\bm{y}$};
				\draw [fill] (0.08742572471695986,-0.9961710408648278) circle (2pt);
				\draw (0.2,-1.2) node {$\bm{R_\alpha(y)}$};
				\draw [fill] (0.6084388609788624,0.7936007512916963) circle (2pt);
				\draw (0.85,0.95) node {$\bm{R^2_\alpha(y)}$};
				\draw [fill] (-0.9847134853154287,-0.1741819503793116) circle (2pt);
				\draw (-1.35,-0.18) node {$\bm{R^3_\alpha(y)}$};
				\draw [fill] (0.843755294812397,-0.5367280526263231) circle (2pt);
				\draw (1.2,-0.6) node {$\bm{R^4_\alpha(y)}$};
				\draw [fill] (-0.2596043049014895,0.9657150743757781) circle (2pt);
				\draw (-0.2,1.2) node {$\bm{R^5_\alpha(y)}$};
				\draw[rouge,->,line width=1pt,shorten >=1pt] (137:0.8) arc (137:274:0.8);
				\draw[color=rouge] (-0.6,-0.3) node {$\bm{\alpha}$};
			\end{tikzpicture}
		\end{center}
		\caption{\label{Fig:dynamicals_sturmian} Top-left: billiard trajectory. Top-right: digitization of a straight line. Bottom-left: linear flow on the torus. Bottom-right: rotation of the circle.}
	\end{figure}

	\bigskip
	The next theorem establishes the connection between these dynamical and geometrical systems, and Sturmian words.
	
	\begin{theorem}\label{th:sturmian_dynamics}
		Let $w$ be a standard Sturmian word with letter frequencies $(f_1,f_2)$. Then
		\[
		\begin{array}{rcl}
			w &=& w_{\mathrm{billiard}}(0,(f_1,f_2))\\
			&=& w_{\mathrm{line}}(0,(f_1,f_2))\\
			&=& w_{\mathrm{flow}}(0,(f_1,f_2))\\
			&=& w_{\mathrm{rotation}}(f_2,f_2).
		\end{array}
		\]
	\end{theorem}
	
	This theorem provides an alternative proof of the existence of Sturmian words. It also helps develop intuition about their combinatorial properties, and sometimes allows us to establish them, as illustrated in the following exercise.

	\begin{exercise} Let $w$ be a Sturmian word.\\
		1) Prove that if $u$ if a factor of $w$, then $\texttt{mirror}(u)$ (the word $u$ written backwards) is also a factor of $w$.\\
		\noindent 2) [Challenge] Prove the same result using only combinatorial arguments. 
	\end{exercise}

	We conclude this chapter with two remarks and the proof of Theorem \ref{th:sturmian_dynamics}.
	
	\begin{remark}
		\label{rk:minimal_equiv_irratinality} We recall from Theorem \ref{th:connection_CF} that the letter frequencies of a Sturmian word satisfy $f_1/f_2 \notin\Q$. This condition is necessary and sufficient to ensure the minimality of the square billiard, and the minimality of linear flow on the torus. It is also equivalent to asking $f_2 \notin \Q$ (due to the equalities $f_1/f_2 = (1-f_2)/f_2) = 1/f_2 - 1$), and thus, equivalent to asking the minimality of the circle rotation. 
	\end{remark}

	\begin{remark}\label{rk:sturmian_neq_rotation}We almost have the set equality
		\[\{\text{Sturmian words}\} = \{w_{\mathrm{rotation}}(y,\alpha) \; \vert \;  y\in [0,1),\alpha \in [0,1)\backslash \Q\}. \]
		Indeed, the reverse inclusion is true, while the direct inclusion is almost true: for every possible pair of letter frequencies $(f_1,f_2)$, only countably many Sturmian words are missing (see Exercise \ref{exo:all_rotations_are_sturmian} at the end of the chapter).
	\end{remark}

	\begin{proof}[Proof of Theorem \ref{th:sturmian_dynamics}] We justify the second, third and fourth equalities. The first equality will be established by Exercise \ref{exo:sturmian_rotation} below.
		
		Clearly, we will always have $w_{\mathrm{billiard}}(x,\theta)=w_{\mathrm{line}}(x,\theta)$. Indeed, instead of letting it bounce, one can equivalently mirror the billiard table and let the ball continue its trajectory along the same straight line. (The idea of mirroring a billiard table is classical in the theory of \emph{translation surfaces}.) It is also clear that for every pair of parameters $(x, \theta)$, we have $w_{\mathrm{flow}}(x,\theta)=w_{\mathrm{line}}(x,\theta)$: it suffices to take the quotient of the plane, the integer grid, and the half-line by the lattice $\Z^2$.
		
		To establish the fourth equality, we project the linear flow with direction $\theta = (\theta_1,\theta_2)\in \R_{>0}^2$ on the torus $\R^2/\Z^2 = [0,1)^2$ onto a line orthogonal to flow direction, and prove that the resulting one-dimensional dynamical system is an exchange of two intervals or, equivalently, a rotation of the circle. An illustration is provided in Figure \ref{fig:projection}. 
		
		
		\begin{figure}[h!]
		\begin{center}
			\begin{tikzpicture}[line cap=round,line join=round,>=stealth,x=4cm,y=4cm]
				\clip(-0.64,-0.56) rectangle (1.1,1.1);
				\draw [line width=1pt] (0,0)-- (1,0);
				\draw [line width=1.5pt,color=bleu, dash pattern=on 3pt off 3pt on 2pt off 3pt] (1,0)-- (1,1);
				\draw [line width=1.5pt,color=orange2, dash pattern=on 3pt off 3pt on 2pt off 3pt] (1,1)-- (0,1);
				\draw [line width=1pt] (0,1)-- (0,0);
				
				\draw [line width=1.5pt,color=orange2] (-0.5982492470923848,0.6302616315528721)-- (-0.32185604484236374,0.18304803605291403);
				\draw [line width=1.5pt,color=bleu] (-0.32185604484236374,0.18304803605291403)-- (0.12535755065759427,-0.540558761697065);
				
				\draw [line width=0.5pt, color=gris] (0.5,0.31)-- (1,0.6190169943749475);
				\draw [line width=0.5pt, color=gris] (0,0.6190169943749475)-- (0.6164434522374274,1);
				\draw [line width=0.5pt, color=gris] (0.6164434522374274,0)-- (1,0.23705098312484235);
				\draw [line width=0.5pt, color=gris] (0,0.23705098312484235)-- (1,0.8550849718747372);
				\draw [line width=0.5pt, color=gris] (0,0.8550849718747372)-- (0.23447744098732223,1);
				\draw [line width=0.5pt, color=gris] (0.23447744098732223,0)-- (0.5571615021721624,0.1994297174400819);
				\draw [line width=0.5pt,dotted,color=gris] (0.5571615021721624,0.1994297174400819)-- (0.7824576346081474,0.33867038481941825);
				
				\draw [line width=0.5pt,dash pattern=on 2pt off 3pt on 1pt off 3pt,color=gris2] (-0.5982492470923848+5*0.03031,0.6302616315528721+5*0.01874)-- (0,1);
				\draw [line width=0.5pt,dash pattern=on 2pt off 3pt on 1pt off 3pt,color=gris2] (0.12535755065759427+5*0.03031,-0.540558761697065+5*0.01874)-- (1,0);
				\draw [line width=0.5pt,dash pattern=on 2pt off 3pt on 1pt off 3pt,color=gris2] (-0.32185604484236374+5*0.03031,0.18304803605291403+5*0.01874)-- (1,1);
				
				\draw [line width=0.5pt,dash pattern=on 2pt off 4pt,color=gris] (-0.5334412763225085,0.5254001321053025)-- (0,0.8550849718747372);
				\draw [line width=0.5pt,dash pattern=on 2pt off 4pt,color=gris] (-0.4278684673224243,0.3545797388553656)-- (0,0.6190169943749475);
				\draw [line width=0.5pt,dash pattern=on 2pt off 4pt,color=gris] (-0.2570480740724874,0.07818653660534453)-- (0,0.23705098312484235);
				\draw [line width=0.5pt,dash pattern=on 2pt off 4pt,color=gris] (-0.08622768082255047,-0.1982066656446765)-- (0.23447744098732226,0);
				\draw [line width=0.5pt,dash pattern=on 2pt off 4pt,color=gris] (0.01934512817753364,-0.3690270588946134)-- (0.6164434522374274,0);
				\draw [line width=0.5pt,dash pattern=on 2pt off 4pt,color=gris] (-0.1514752650724033,-0.09263385664459245)-- (0.5,0.31);

				\draw [line width=0.5pt,dash pattern=on 2pt off 4pt,color=gris2] (-0.27719683582623883,-0.17033406048948754)-- (-0.1514752650724033,-0.09263385664459245);
				\draw [line width=0.5pt,dash pattern=on 2pt off 4pt,color=gris2] (0.01934512817753364,-0.3690270588946134)-- (-0.19666917952460064,-0.5025312431108105);
				\draw [line width=0.5pt,dash pattern=on 2pt off 4pt,color=gris2] (-0.6438827750245585,0.2210755546391683)-- (-0.4278684673224243,0.3545797388553656);
				
				\draw [->,line width=1pt,color=red] (0.5,0.31) -- (0.8031851677125179,0.4973787385311733);
				\draw [->,line width=1.5pt,color=bleu] (-0.6099178269930644,0.24206704694875553) -- (-0.16270423149310653,-0.4815397508012234);
				\draw [->,line width=1.5pt,color=orange2] (-0.2504761444219233-0.03031,-0.1538197649987234-0.01874) -- (-0.5268693466719443-0.03031,0.2933938305012346-0.01874);
				
				\draw [fill] (0.5,0.31) circle (2pt);
				\draw (-0.6,0.5) node {$\bm{z_4}$};
				\draw [fill] (-0.5334412763225085,0.5254001321053025) circle (2pt);
				\draw (-0.5,0.37) node {$\bm{z_1}$};
				\draw [fill] (-0.4278684673224243,0.3545797388553656) circle (2pt);
				\draw (-0.33,0.06) node {$\bm{z_3}$};
				\draw [fill] (-0.2570480740724874,0.07818653660534453) circle (2pt);
				\draw (-0.23,-0.08) node {$\bm{z_0}$};
				\draw [fill] (-0.08622768082255047,-0.1982066656446765) circle (2pt);
				\draw (-0.16,-0.2) node {$\bm{z_5}$};
				\draw [fill] (0.01934512817753364,-0.3690270588946134) circle (2pt);
				\draw (-0.06,-0.36) node {$\bm{z_2}$};
				\draw [fill] (-0.1514752650724033,-0.09263385664459245) circle (2pt);
				
				\draw [line width=1pt] (-0.5982492470923848-0.03031,0.6302616315528721-0.01874)-- (-0.5982492470923848+0.03031,0.6302616315528721+0.01874);
				
				\draw [line width=1pt] (0.12535755065759427-0.03031,-0.540558761697065-0.01874)--(0.12535755065759427+0.03031,-0.540558761697065+0.01874);
				
				\draw [line width=1pt] (-0.32185604484236374-0.03031,0.18304803605291403-0.01874)--(-0.32185604484236374+0.03031,0.18304803605291403+0.01874);
				
				\draw (0.2,-0.48) node {\small{$\bm{0}$}};
				\draw (-0.16,0.23) node {\small{$\bm{\ell-\beta}$}};
				\draw (-0.52,0.69) node {\small{$\bm{\ell}$}};
				
				\draw (0.42,0.35) node {$\bm{x}$};
				\draw (0.6,0.45) node {\textcolor{red}{$\bm{\theta}$}};
			\end{tikzpicture}
		\end{center}
		\caption{From the linear flow on the 2D torus to a rotation on a circle.}\label{fig:projection}
		\end{figure}
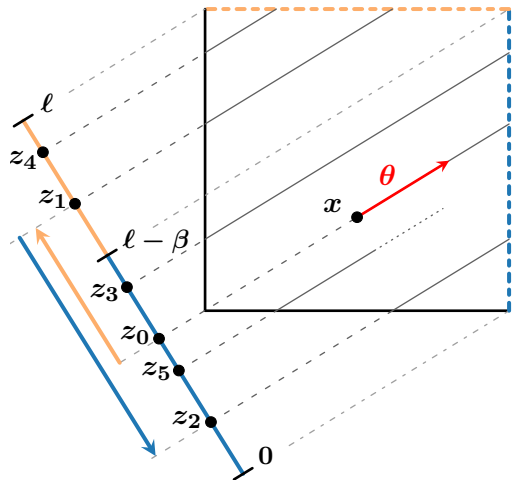

		Observe that the image, after the projection, of \\
		- the unit square $[0,1)^2$ is an oblique segment (its length will be denoted by $l$);\\
		- the top edge of the square is the orange/light gray upper interval of the segment (its length will be denoted by $\beta$);\\
		- the right edge of the square is the blue/dark gray lower interval of the segment (its length is $l-\beta$);\\
		- the trajectory of a point $x\in [0,1)^2$ under the action of the flow on the torus is a sequence of points on the segment, denoted by $(z_n)_{n\in \N_{>0}}$.
		
		\medskip
		We identify the oblique segment with $[0,l)$, with $0$ at the bottom endpoint. We assert that: \\
		1) The sequence $(z_n)_{n}$ is the trajectory of $z_1$ for the rotation by angle $\beta$ on the length-$l$ circle $\R/ l\Z = [0,l)$. \\
		2) The flow word $w_{\mathrm{flow}}(0,\theta)$ is equal to the rotation word $w_{\mathrm{rotation}}(y, \alpha)$, with $y = \alpha= \frac{\theta_2}{\theta_1+\theta_2}$.
		
		\medskip
		
		The first claim follows from the observation that the next section crossed by a point $x$ in the torus is vertical if, and only if, its projection $z=:z_k$ belongs to the blue/dark gray interval of the segment (that is, $z_k \in [0,1-\beta)$). In this case, $z_{k+1} = z_{k} + \beta$. In the other case (that is, when the next crossed section is horizontal, or equivalently, when $z_k \in [1-\beta,1)$, we have $z_{k+1} = z_k + \beta - l$. 
		In both cases, $z_{k+1} = z_k + \beta \mod l$, and the first claim is proven.
		
		Next, we rescale the segment by a factor $1/l$, in order to obtain a rotation on the unit circle. Denote $\alpha = \beta/l$ and $y_k = z_k/l$. Then we have that for all $k \in \N_{>0}$
		\[w_{\mathrm{flow}}(x,\theta)[k]= \begin{cases} 1 \text{ if } z_k \in [0,1-\alpha)\\
			2 \text{ otherwise}, \end{cases}\]
		thus establishing that $w_{\mathrm{flow}}(x,\theta) = w_{\mathrm{rotation}}(z_0,\alpha)$.
		We conclude the proof by calculating that
		\[\alpha =  \frac{\beta}{l}  = \frac{\theta_2}{\theta_1+\theta_2},\]
		and that the projection of the point $(0,0)$ of $[0,1)^2$ is $\alpha$. \end{proof}
	
	\begin{exercise} \label{exo:sturmian_rotation} Let $\alpha \in [0,1]\setminus\Q$. The aim of the exercise is to establish that $w_{\mathrm{rotation}}(\alpha,\alpha)$ is the standard Sturmian word with letter frequencies $(f_1,f_2)=(1-\alpha, \alpha)$. This equality completes the proof of Theorem \ref{th:sturmian_dynamics}.\\
		1) Let $y \in [0,1)$. Justify that the letter frequencies of the rotation word $w_{\mathrm{rotation}}(y,\alpha)$ exist and are equal to $(1-\alpha,\alpha)$. \\
		2) Let $y\in[0,1)$. Prove that $w_{\mathrm{rotation}}(y,\alpha)$ is a Sturmian word.\\
		3) Prove that a Sturmian word is standard if and only if all its finite prefixes are left-special. \\
		4) Prove that all the finite prefixes of the rotation word $w_0 = w_{\mathrm{rotation}}(\alpha,\alpha)$ are left-special.\\
		5) Conclude.
	\end{exercise}







	
	\begin{exercise}\label{exo:all_rotations_are_sturmian}
		The aim of the exercise is to prove Remark \ref{rk:sturmian_neq_rotation}. 
		Let $\alpha\in\R\setminus\Q$.\\
		1) Let $x,y\in[0,1)$. Prove that if $x\neq y$, then $w_{\mathrm{rotation}}(x,\alpha)\neq w_{\mathrm{rotation}}(y,\alpha)$.\\
		2) Denote by $w_0$ the standard Sturmian word with letter frequencies $(1-\alpha,\alpha)$. Prove that $w=2\cdot w_0$ is a Sturmian word, and yet, is not a rotation word according to our definition. \emph{Hint:} we proved in Exercise \ref{exo:sturmian_rotation} that $w_0 = w_{\mathrm{rotation}}(\alpha,\alpha)$.   
	\end{exercise}

	
	
	
	

	\section{Commented bibliography}
	
	The reader interested in learning more about Sturmian words is invited to consult:
	\begin{itemize}
		\item the historical papers by M. Morse and G. Hedlund from the 1940s \cite{MH38,MH40},
		\item the historical papers by C. Coven and G. Hedlund from the 1970s \cite{CH73a,CH73b},
		\item the two dedicated book chapters from  the 2000s: \cite[Chapter 2]{Loth02} and \cite[Chapter 6]{Pyth02}.
	\end{itemize}

	The introduction to Sturmian words that we gave in this lecture differs from the presentations in the two book chapters mentioned above. Indeed, we established their connection to continued fractions and to geometrical and dynamical systems \emph{directly} from their definition as words with complexity  $n\mapsto n+1$. We did not use abelian combinatorics on words (in particular, we did not use \emph{balancedness}).

\cleardoublepage
	\mychapter{Towards a generalization of the Morse--Hedlund theorem to $d$-ary alphabets: a theorem by Tijdeman (1999)}{Towards a generalization of the Morse--Hedlund theorem to $d$-ary alphabets: a theorem by Tijdeman (1999)}{Towards a generalization of the Morse--Hedlund theorem to $d$-ary alphabets }\label{sec:Generalization_MH38}
	
	\section{Motivations}\label{sect:motivations_d-ary}
	The first chapter of this course was devoted to the study of Sturmian words, which form a remarkable class of  binary infinite words. As a brief reminder:
\begin{itemize}
\item The \emph{complexity} of an infinite word $w \in \mathcal{A}^{\N}$ is the function $p_w:\N \rightarrow \N$ that counts, for each integer $n$, the number of factors (i.e., subwords) of length $n$ in $w$. 
\item A celebrated theorem by M. Morse and G. Hedlund from 1938 asserts that an infinite word is eventually periodic if and only if there exists an integer $n$ for which $p_w(n)\leq n$. Therefore, every non eventually-periodic word must satisfy $p_w(n)\geq n+1$ for all $n$.
\item There exist infinite words whose complexity is exactly $p_w(n)= n+1$ for all $n$; they are called \emph{Sturmian words}. 
\item Sturmian words are written over the binary alphabet (indeed, they satisfy $p_w(1)=2$). They exhibit remarkable interactions with arithmetic, dynamics, and geometry. More precisely, they are intimately connected to the \emph{continued fraction},  and can equivalently be defined as the symbolic codings of irrational rotations, irrational trajectories on a square billiard table, minimal exchanges of two intervals, and digitizations of straight lines with an irrational slope.
\end{itemize}

Therefore, it is natural to wonder to what extent these satisfying results can be extended to larger alphabets.
	
	\begin{question}\label{question:gene_MH} Let $d\geq 3$. Does there exist a $d$-ary analogue of the Morse--Hedlund theorem, that is, a theorem guaranteeing that, below a given complexity bound, infinite $d$-ary words are ``trivial''?\end{question}
	\begin{question}\label{question:understand_d_ary_minimal}If so, can we dynamically and arithmetically characterize, when $d\geq 3$,  the set of non-trivial $d$-ary infinite words with minimal complexity? \end{question}
	
	The aim of this second chapter is to convince the reader that the Theorem of Tijdeman provides a very satisfying answer to Question \ref{question:gene_MH}, and that Question \ref{question:understand_d_ary_minimal}, which is essentially open at the time of writing, deserves further investigation. 


	\section{Examples of $d$-ary generalizations of Sturmian words}
	
	Sturmian words can be generalized to $d$-ary alphabets in multiple ways:
	\begin{enumerate}
		\item[1)] From a combinatorial point of view, by imposing conditions on the set of special factors, or as limits of compositions by specific substitutions over the $d$-letter alphabet; 
		\item[2)] From an arithmetic point of view, as words generated by $d-1$-dimensional continued fraction algorithms (note that there exist several equally legitimate choices for the continued fraction, which give rise to several classes of $d$-ary words);
		\item[3)] From a geometric or dynamical point of view, as the coding of billiard trajectories on a polygonal or hypercubic table, or as the coding of a rotation of the torus $\mathbb{R}^{d-1}/\mathbb{Z}^{d-1}$, or as the coding of an exchange of $d$ intervals, \emph{etc}.
	\end{enumerate}
	
	These constructions give rise to various families of well-studied $d$-ary infinite words, such as  \emph{Arnoux--Rauzy words} \cite{AR91}, \emph{episturmian words} \cite{DJP01,JP02,GJ09}, \emph {words associated with the Cassaigne--Selmer continued fraction algorithm} \cite{CLL22}, \emph{dendric words} \cite{BDDLP15}, \emph{S-adic words} \cite[Chapter 12]{Pyth02},\cite{BD14}, \emph{natural codings of interval exchange transformations} \cite{Kea75,Yoc05,FZ08,Del16}, \emph{hypercubic billiard words}\cite{AMST94,Bar95,Bed09}, etc.
	They are all considered as generalizations of Sturmian words.
	
\medskip
	Interestingly, these $d$-ary generalizations of Sturmian words do not all have the same complexity function. For example, on the ternary alphabet, the complexity of Arnoux--Rauzy and Cassaigne--Selmer words is $n\mapsto 2n+1$, the complexity of non Arnoux--Rauzy episturmian words  is strictly lower than $n\mapsto 2n+1$, and the complexity of almost all cubic billiard words is $n\mapsto n^2+n+1$.


	\section[Discussion on what ``minimal complexity'' could mean for $d$-ary words]{What is the minimal complexity for a $d$-ary word? (Discussion on what ``minimal'' should mean.)}
	
	From now on, and throughout Chapters 2 and 3, when we say that an infinite word $w$ is a ``$d$-ary infinite word'', we mean that  $d$ distinct letters actually appear in $w$. Specifically, we require that $p_w(1)=d$. 
	
	\bigskip

	The next theorem asserts that the constant function $p:n\mapsto d$ is the absolute minimal complexity for a $d$-ary infinite word.
	
	\begin{theorem}\label{th:absolute_minimum}
		Let $d\geq 1$. Let $w\in\{1,\ldots,d\}^\N$ be an infinite word with $p_w(1)=d$. Then $p_w(n)\geq d$ for every $n\in\N_{>0}$. Moreover, there exists a $d$-ary infinite word $w$ such that $p_w(n)=d$ for all $n\in\N_{>0}$.
	\end{theorem}
	
	\begin{proof}
		The first assertion follows from the equality $p_w(1)=d$ and the fact that complexity $p_w$ is a non-decreasing function. For the second assertion, it suffices to check that the periodic word $w=(12\ldots d)^\omega$ is a $d$-ary infinite word with the desired complexity.
	\end{proof}

	However, Theorem \ref{th:absolute_minimum} is of limited interest: by taking $d=2$, we do not retrieve the Morse--Hedlund theorem  that we wish to extend: ``the minimal complexity function for a non eventually-periodic binary word is $n\mapsto n+1$''. 
	In our research for a generalization of this theorem, we need at least to keep restricting ourselves to non eventually-periodic $d$-ary words. As we will see in the following paragraphs, this restriction is not even sufficient.

	\medskip
	The theorem below asserts that the linear function $n\mapsto n+d-1$ is the minimal complexity for a non eventually-periodic $d$-ary word. 
	
	\begin{theorem} Let $d
		\geq 2$. Let $w\in\{1,\ldots,d\}^\N$ be a non eventually-periodic word with $p_w(1)=d$. Then $p_w(n)\geq n+d-1$ for every $n\in\N_{>0}$. Moreover, there exists a non eventually-periodic $d$-ary word $w$ such that $p_w(n)=n -d-1$ for all $n\in\N_{>0}$. 
	\end{theorem}
	
	\begin{proof} Since $w$ is non eventually-periodic, it must admit at least one right-special factor of each length; therefore, its complexity function is strictly increasing. Since $p_w(1)=d$, an immediate induction yields $p_w(n)\geq n+d-1$ for every $n\geq 1$.
		
		For the second assertion, let $w_0$ be a Sturmian word over the binary alphabet $\{d-1,d\}$, and consider the $d$-ary infinite word obtained as the concatenation $w:=12\ldots(d-2)w_0$. Since $w_0$ is Sturmian, $w$ is non-eventually periodic and has the desired complexity function $n\mapsto n+d-1$. Indeed, for every $n\geq 1$, the word $w$ admits $d-2$ factors that start with the letters $1, 2, \ldots, d-2$, respectively, and $n+1$ factors that stem from the Sturmian word $w_0$, and are written with the sole letters $d-1$ and $d$.
	\end{proof}
	
	The reader may be disappointed by the construction made in the proof above: we simply picked a Sturmian word, and prefixed it by some additional letters in order to make it fit the definition of ``$d$-ary word''. These additional letters do not even reappear in $w$. We are far away from a satisfactory generalization of Sturmian words.

	To exclude these cumbersome examples, we may require 
	that every letter occurs infinitely many times in $w$. We can even take an extra precaution and require that \emph{every factor} occurs infinitely many times in $w$, with bounded gaps. We recall that this property is called \emph{uniform recurrence}, and that Sturmian words are uniformly recurrent (establishing this property was the aim of Exercise  \ref{exo_UR} in Chapter 1).

	\begin{question} Does there exist a non-eventually periodic, uniformly recurrent $d$-ary word $w$, with complexity $n\mapsto n+d-1$? \end{question}
	
	The answer is `yes', as evidenced by the following exercise.
	
	\begin{exercise}\label{exo:disguised_sturmian_word}
		Let $w_0 \in \{1,2\}^{\N}$ be a Sturmian word. Denote by $w$ the image of $w_0$ by the substitution $\sigma : 1 \mapsto 31, 2\mapsto 32.$\\
		1) Prove that $w$ is a non eventually-periodic, uniformly recurrent ternary word.\\
		2) Prove that the complexity of $w$ is $p_w:n\mapsto n+2$. (Hint: prove that $p_w(n+1)-p_w(n)=1$ by studying right-special factors of $w$.) \\
		3) Generalize the construction to the $d$-letter alphabet, for $d\geq 4$.
	\end{exercise}
	
	Again, the construction made in Exercise \ref{exo:disguised_sturmian_word} is  frustrating: the word $w$ is somewhat \emph{too poor} to be considered as a generalization of Sturmian words.
	
	\begin{question} Does there exist a non-eventually periodic, uniformly recurrent word, with complexity $n \mapsto n+d-1$, which is not an artificial distortion of Sturmian words?
	\end{question}
	
	The answer is `no'. The next theorem asserts that  all non eventually-periodic, uniformly recurrent $d$-ary words with complexity $n \mapsto n+d-1$ can be constructed as images of Sturmian words by substitutions similar to that of Exercise \ref{exo:disguised_sturmian_word}.

	\begin{theorem}[S. Ferenczi and C. Mauduit, 1997; see \cite{FM97}, Lemma 4] \label{th:quasi_Sturmians}
		Let $d\geq 3$, and $w$ be an infinite word. The word $w$ is recurrent, with complexity $n\in\N_{>0}\mapsto n+d-1$ if, and only if, there exist:
		\begin{itemize}
			\vspace{-0.3cm}
			\item[-] an integer $m\in\N$,
			\vspace{-0.3cm}
			\item[-] a Sturmian word $w_0\in\{a,b\}^\N$,
			\vspace{-0.3cm}
			\item[-] a partition of the alphabet $\{1,\ldots,d\}=\mathcal{A}\sqcup\mathcal{B}\sqcup\mathcal{C}$ where $\mathcal{A}=\{a_1,\ldots,a_{N_{\mathcal{A}}}\}$, $\mathcal{B}=\{b_1,\ldots,b_{N_{\mathcal{B}}}\}$ and $\mathcal{C}=\{c_1,\ldots,c_{N_{\mathcal{C}}}\}$, with both $\mathcal{C}\neq\emptyset$ and $\mathcal{A}\sqcup\mathcal{B}\neq\emptyset$,
		\end{itemize}
		\vspace{-0.1cm}
		such that $w=S^m\sigma(w_0)$, where $S$ denotes the shift operator and $\sigma$ is the substitution defined by
		\[
		a\mapsto c_1\ldots c_{N_{\mathcal{C}}}a_1\ldots a_{N_{\mathcal{A}}}, \hspace{0.5cm} b\mapsto c_1\ldots c_{N_{\mathcal{C}}}b_1\ldots b_{N_{\mathcal{B}}}.
		\]
		Furthermore, in this case, the word $w$ is uniformly recurrent.
	\end{theorem}
	
	(It seems that this theorem has been previously established for ternary alphabets by P. Alessandri in \cite{Ale95}, see also \cite{KT07}.)
	
	Therefore, in the literature, the infinite words with complexity $n \mapsto n+d-1$, and more generally, those whose complexity function is eventually equal to $n+k$, for a certain constant $k$, are called ``quasi-sturmian words'' \cite{Cas98}.


	\section{The condition of rational independence of letter frequencies.}
	
	In the previous section, we saw that the non
	eventual-periodicity condition, which is very efficient to exclude trivial words on the binary alphabet, is too weak to eliminate them on larger alphabets. 
	The aim of this section is to convince the reader that, on the $d$-letter alphabets, for arbitrary $d\geq 1$, a good condition to ask is the \emph{rational independence} of the letter frequencies.
	\medskip
	
	Let us recall that the real numbers $x_1, \ldots, x_d$ are ``rationally independent'' when none of them can be written as a linear combination of the others with rational coefficients or, equivalently, when the dimension of the linear space they spanned over the field $\Q$ is $d$.
	The conviction that, when dealing with arbitrary size alphabets, the condition of rational independence of letter frequencies should replace the non eventually-periodicity  is motivated by the  following four facts.

	\paragraph{Fact 1. The interesting part of the Morse--Hedlund theorem (namely: an infinite word with a low complexity is eventually periodic) can be formulated in terms of rational dependence of the letter frequencies.} Specifically, we have:
	
	\begin{proposition}\label{prop:MH_equiv_RI} Let $w \in \{1,2\}^{\N}$. If there exists an integer $n\geq 1$ such that $p_w(n)\leq n$, then $w$ has rationally dependent letter frequencies. \end{proposition}
	
	\begin{proof}
		Under this assumption, the Morse--Hedlund theorem (Chapter 1, Theorem \ref{MH38}) guarantees that the word $w$ is eventually periodic; let us write it $w= u\cdot p^{\omega}$. One readily checks that the frequencies of the letters $1$ and $2$ in $w$ are equal to those in its periodic suffix $p^{\omega}$, that is $f_1=\vert p \vert_1/\vert p \vert$ and $f_2=\vert p \vert_2 / \vert p \vert$. These (rational) numbers are obviously rationally dependent.
	\end{proof}
	
	\paragraph{Fact 2. Sturmian words have rationally independent letter frequencies, and it is one of their central properties.} We recall from Chapter \ref{sec:MH38_Sturmian} that this property first emerges through their connection with the continued fraction. Specifically, Theorem \ref{th:connection_CF} asserts that the ratio of the letter frequencies $f_1/f_2$ of any Sturmian word is an irrational number, which is equivalent to asserting that $f_1$ and $f_2$ are rationally independent. This property further manifests itself in all geometrical and dynamical characterizations of Sturmian words, as highlighted by Remark \ref{rk:minimal_equiv_irratinality}: Sturmian words encode \emph{minimal} rotations of the circle, \emph{minimal} square billiard trajectories, \emph{minimal} linear flows on the two-dimensional torus, and half-lines of the plane with an \emph{irrational} slope. 
	
	\medskip
	Finally, it is worth noticing that, by consequence of their definition and of Proposition \ref{prop:MH_equiv_RI}, Sturmian words are, among the set of binary words with rationally independent letter frequencies, exactly those with minimal complexity.
	
	\paragraph{Fact 3. Classes of words which are commonly considered as good combinatorial or dynamical generalizations of Sturmian words also have rationally independent letter frequencies.} 
	
	It is the case, for example, of the words encoding minimal trajectories in an hypercubic billiard table (see, for example, \cite{AV24}), but also of Arnoux--Rauzy words \cite{And21,DHS22}, strict episturmian words \cite[Chapter 4]{And21thesis}, and Cassaigne--Selmer words \cite{CLL17,CLL22}.
	
	\paragraph{Fact 4. No $d$-ary quasi-sturmian word, for $d\geq 3$, has rationally independent letter frequencies.}
	
	It will be a consequence of the theorem of Tijdeman presented in the next section.
	In the particular case of recurrent $d$-ary words with complexity $n \in \N_{>0} \mapsto n+d-1$, this fact can alternatively be derived from Theorem \ref{th:quasi_Sturmians}, as in the exercise below.

	\begin{exercise}
		Let $d\geq 3$. Let $w$ be a recurrent $d$-ary word with complexity $n \in \N_{>0}\mapsto n+d-1$. Prove that $w$ admits letter frequencies, and that these letter frequencies span a rational linear space of dimension $2$. 
	\end{exercise}

	Now that the \emph{rational independence of the letter frequencies} appears as a reasonable replacement for the  \emph{non eventually-periodic} condition when the alphabet is of arbitrary size, a new question arises.
	
	\begin{question}Let $d\geq 3$. What is the minimal complexity for a $d$-ary word with rationally independent letter frequencies? \end{question}
	
	This question was fully answered by R. Tijdeman in 1999.
	
	\section{A theorem by Tijdeman}
	
	Unfortunately, Tijdeman's answer has remained mostly unnoticed (it is for example missing in \cite{BB13}) until 2022, when it was rediscovered by J. Cassaigne and the author of these lecture notes. \medskip
	
	Tijdeman's theorem states that the minimal complexity for an infinite $d$-ary word with rationally independent letter frequencies is the linear function $n\mapsto (d-1)n+1$

	\begin{theorem}[Tijdeman, 1999, \cite{Tij99}]\label{th:Tijdeman}
		Let $d\geq 1$. Let $w\in\{1,\ldots,d\}^\N$ be an infinite word with $p_w(1)=d$ and rationally independent letter frequencies. Then $p_w(n)\geq (d-1)n+1$ for every $n\in\N_{>0}$. Moreover, there exists a $d$-ary infinite word $w$, with rationally independent letter frequencies, and such that $p_w(n)=(d-1)n+1$ for all $n\in\N_{>0}$.
	\end{theorem}

Of course, when $d=2$, we retrieve the Proposition \ref{prop:MH_equiv_RI} that motivated our investigation.

	\section{Towards a characterization of $d$-ary infinite words with minimal complexity in Tijdeman's sense.}\label{sect:dary_min_complexity_characterization}

The aim of this subsection is to initiate the study of infinite words that achieve Tijdeman's minimal complexity. First, we claim that they are all recurrent.

\begin{proposition}\label{prop:tijdeman_recurrent}
	Let $d\geq 1$, and let $w$ be an infinite word with complexity $p_w(n)=(d-1)n+1$ and rationally independent letter frequencies. Then $w$ is recurrent.
\end{proposition}

\begin{proof}
	Let $w$ be an infinite word  with complexity $p_w(n)=(d-1)n+1$ and rationally independent letter frequencies $f_1,\ldots,f_d$. We proceed by contradiction. Assume that $w$ is not recurrent, and denote by $u$ a factor that appears finitely many times in $w$. Let $w'$ be a suffix of $w$ that does not contain the factor $u$. Observe that $w'$ is still an infinite $d$-ary word (indeed, if a letter is missing in $w'$, then its frequency in $w$ was $0$, which is impossible), and that the letter frequencies of $w'$ are exactly those of $w$ (indeed, the frequencies of the letters do not depend on a finite prefix). Furthermore, for $n=\vert u\vert$, we have $p_{w'}(n)\leq p_w(n) -1<(d-1)n+1$. Therefore, $w'$ is an infinite $d$-ary word with rationally independent letter frequencies and a complexity smaller than the minimal bound given by Tijdeman's Theorem \ref{th:Tijdeman}. We have reached a contradiction. 
\end{proof}

It is natural to ask which classical classes  of infinite words are optimal for Tijdeman's theorem
	
	\begin{question} Let $d\geq 3$. Which $d$-ary infinite words have rationally independent letter frequencies, and complexity $p_w:n\in\N_{>0}\mapsto (d-1)n+1$? \end{question}
	
	We recall that for $d=1$ and $d=2$, these classes of words are completely understood.
		\begin{itemize}
		\item For $d=1$, it is the single word  $\{w=111111111111...\}$. Indeed, it is the unique $1$-ary word, and it trivially has rationally independent letter frequencies and the desired complexity $n \mapsto 1$.
		\item For $d=2$, these are exactly Sturmian words. Indeed, Sturmian words are, by definition, the infinite words with complexity $n\mapsto n+1$; moreover, their letter frequencies are rationally independent by Theorem \ref{th:connection_CF}.
		\end{itemize}
		By contrast, to this day, only partial answers are known when $d\geq 3$. 
	
	\begin{itemize}
	\item For $d=3$, we know that the class of words with Tijdeman's minimal complexity contains at least all Arnoux--Rauzy words, all Cassaigne--Selmer words, and almost all codings of interval exchange transformations on three intervals: indeed, these are ternary words with complexity $n\mapsto 2n+1$, and rationally independent letter frequencies.

	\item More generally, for arbitrary alphabet size $d$, the class of words with Tijdeman's minimal complexity contains at least all strict episturmian words, and almost all codings of interval exchange transformations on $d$ intervals.
	\end{itemize}

	 Interestingly, all the aforementioned classes of words (namely: Sturmian, Arnoux--Rauzy, Cassaigne--Selmer and strict episturmian words, but also the  codings of regular interval exchange transformations) belong to a meta-class of infinite words known as \emph{dendric words} (which will be defined in the next chapter). 
	 In fact, it will be a by-product of our alternative proof of Tijdeman's theorem that all infinite words with minimal complexity in Tijdeman's sense are dendric.
	 
\begin{theorem}[Andrieu, Cassaigne, 2022, \emph{unpublished}]\label{th:dendric} Let $d\geq 1$. Every infinite $d$-ary word with rationally independent letter frequencies and complexity $n\mapsto (d-1)n+1$ is dendric.
	\end{theorem}

Since every $d$-ary dendric word has complexity $n \mapsto (d-1)n+1$ (see \cite{BDDLP15} or Proposition \ref{prop:charac_dendric} in Chapter 3), we deduce the following characterization.

\begin{corollary}
 Let $d\geq 1$. The set of infinite $d$-ary words with rationally independent letter frequencies and complexity $n\mapsto (d-1)n+1$, and the set of $d$-ary dendric words with rationally independent letter frequencies, coincide.
\end{corollary}
	 
	 Unfortunately, not all dendric words have rationally independent letter frequencies. Notably, some codings of regular $d$-interval exchange transformations do not: see, for example,\cite[Example 3.3]{Via06}. Therefore, the question of characterizing infinite $d$-ary words with Tijdeman's minimal complexity, in the spirit of the elegant combinatorial, geometrical and dynamical characterizations of Sturmian words remains an open problem for $d\geq 3$ at the time of writing.

	\section{The original and the strengthened versions of Tijdeman's theorem}\label{ssect:tijdeman_complet_t_ameliore}

	The original theorem by Tijdeman is, in fact, stronger than the Theorem \ref{th:Tijdeman} that we stated. It formalizes and \emph{quantifies} the idea according to which, if the complexity of an infinite word $w$ does not grow fast enough, then $w$ is \emph{algebraically poor}. In this last section, we present the complete, original version of Tijdeman's theorem, and we introduce a new, strengthened version of it. 
	
	\medskip

	To state the original form of Tijdeman's theorem, in which the existence of the letter frequencies is not even required, we first need to introduce some notation. Given a $d$-ary infinite word $w$, we consider the sequence
	\[
	\Big(\frac{1}{n}(|\pref_n(w)|_1,\ldots,|\pref_n(w)|_d)\Big)_n
	\]
	which lives in the compact set $[0,1]^d$. If this sequence converges, the limit is, by definition, the frequencies of the letters in $w$. If not, we denote by $F_w$ the set of all its subsequential limits:
	\begin{equation}\label{eq:subfrequencies}
		F_w:=\Big\{(f_1,\ldots,f_d)\in[0,1]^d \;\Big|\; \begin{array}{c}\text{there exists an increasing sequence $(n_l)_l\subseteq\N$ s.t.} \\ \forall\,i\in\{1,\ldots,d\}, \; |\pref_{n_l}(w)|_i/n_l\underset{l\to\infty}{\longrightarrow} f_i \end{array}\Big\}.
	\end{equation}
	
	\noindent We now define the integer $\delta_w \in\{1,\ldots,d\}$ as the \emph{minimal degree of irrationality} of the letters in $w$:
	\[
	\begin{array}{c}
		\ds\delta_w:=\min \big\{\dim\Span_\Q(f_1,\ldots,f_d) \,|\,(f_1,\ldots,f_d) \in F_w\big\}.
	\end{array}
	\]

	\begin{theorem}[Tijdeman 1999, original version]\label{th:Tij_optimal}
		Let $w$ be an infinite $d$-ary infinite word. Denote by $\delta_w$ the minimal degree of irrationality of the letters in $w$. Then, for every $n\geq 1$,
		\[
		p_w(n)\geq (\delta_w-1)(n-1)+d.
		\]
	\end{theorem}
	
	In fact, Tijdeman's original theorem can be strengthened by replacing the minimal degree $\delta_w$ of irrationality of the letters by the \emph{maximal degree} of irrationality of the letters:
	\begin{equation}\label{eq:max_degree_irrationality_def}
		\begin{array}{c}
			\ds\Delta_w:=\max\big\{\dim\Span_\Q(f_1,\ldots,f_d) \,|\,(f_1,\ldots,f_d)\in F_w\big\}.
		\end{array}
	\end{equation}
	In the Exercise \ref{exo:delta_neq_Delta} below, we exhibit an infinite word $w$ for which the minimal and maximal degrees of irrationality of the letters differ.

	\begin{theorem}[Andrieu, Cassaigne, 2022, \emph{unpublished}]\label{th:Tij_ameliore}
		Let $w$ be an infinite $d$-ary word. Denote by $\Delta_w$ the maximal degree of irrationality of the letters in $w$. Then, for every $n\geq 1$,
		\[
		p_w(n)\geq (\Delta_w-1)(n-1)+d.
		\]
	\end{theorem}
	
	Theorem \ref{th:Tij_ameliore} will be established in Chapter 3.
	
	\medskip
	
	We conclude this chapter by two remarks and an exercise.
	
	\begin{remark}
		If $w$ admits frequencies of letters, and if furthermore these frequencies are rationally independent, both Theorems \ref{th:Tij_optimal} and \ref{th:Tij_ameliore} boil down to Theorem~\ref{th:Tijdeman}. Indeed, in this case, we have $\Delta_w = \delta_w = d$, from which we immediately derive \[(\Delta_w-1)(n-1)+d = (d-1)n+1.\]
	\end{remark}


	\begin{remark}
		It will follow from its proof that Theorem~\ref{th:Tij_ameliore} can be equivalently stated as follows: ``if there exists an integer $k\geq 1$ and a length $m\geq 0$ such that $p_w(m+1)-p_w(m)\leq k-1$, then $\Delta_w\leq k$'', or equivalently, ``$\Delta_w\leq\min_{m\geq 0} \,\{p_w(m+1)-p_w(m)+1\}$''. 
		This formulation is to be understood as the $d$-ary counterpart of ``a word is eventually periodic if and only if its complexity function is locally constant'' which is, as we discussed in Remark \ref{rk:MH38} of Chapter 1, equivalent to the Morse--Hedlund theorem.
	\end{remark}
	
	\begin{exercise}[Construction of an infinite binary word for which $\delta_w<\Delta_w$]\label{exo:delta_neq_Delta} Let $(u_n)$ be the sequence of finite words constructed as follows. Initially, $u_0= 1$. Then, for every even integer $n$, we concatenate, at the end of $u_n$, as many occurrences of $0$ as needed, so that $u_{n+1}$ contains the same numbers of both letters $0$ and $1$ (this will always be possible since for every even integer $n$, $u_n$ contains more $1$s than $0$s). For every odd integer $n$, we concatenate $n\times \vert u_n\vert$ occurrences of $1$ at the end of $u_n$, so that $u_{n+1}$ is $n+1$ times longer than $u_n$. We thus obtain $u_0 = 1$, $u_1 = 10$, $u_2=1011$, $u_3=101100$, $u_4=101100111111111111111111$, etc. By construction, the sequence of finite words $(u_n)$ converges to an infinite word, which we denote by $w$:
		\[w = 101100111111111111111111000000000000000000...\]\\
		1) Prove that $w$ does not admit frequencies of letters, and that both $(0,1)$ and $(1/2,1/2)$ belong to $F_w$ (as defined in \eqref{eq:subfrequencies}).\\
		2) Prove that actually $F_w=\{(t,1-t): 0\leq t \leq 1/2\}$. Deduce that $\delta_{w} = 1$ and $\Delta_{w} = 2$.
		
		\medskip
		
		Now, we replace, one by one, each occurrence of $1$ in $w$ by the letters successively occurring in the Fibonacci word
		$w_{Fibo}= 1211212112112121121211211212112112...$
		(We recall that the Fibonacci word was defined in Example \ref{ex:def_fibo}, and that we established that its letter frequencies exist and satisfy $f_1/f_2 = \varphi$, where $\varphi$ denotes the golden ratio.)\
		We thus obtain an infinite word written with the letters $0$, $1$ and $2$:
		\[w' =
		102100121211211212112121000000000000000000...\]
		3) Prove that $F_{w'}=\big\{\big(t,(1-t)f_1,(1-t)f_2\big): 0\leq t \leq 1/2\big\}$. Deduce that $\delta_{w'} = 2$ and $\Delta_{w'} = 3$. Check that Theorem \ref{th:Tij_optimal} only ensures that $p_{w'}(n)\geq n+2$, whereas Theorem \ref{th:Tij_ameliore} guarantees that $p_{w'}(n)\geq 2n+1$.
	\end{exercise}


		
	\mychapter{An alternative, algebraic proof of Tijdeman's theorem and its consequences (2022)}{An alternative, algebraic proof of Tijdeman's theorem and its consequences (2022)}{An alternative, algebraic proof of Tijdeman's theorem and its consequences }\label{sec:Proof_Tijdeman}
	
	\section{Introduction}
	In this third and last chapter, we present an alternative proof of Tijdeman's original theorem from 1999. This proof was obtained, but unpublished, by J. Cassaigne and the author in 2022. It actually establishes a slightly stronger result than Tijdeman's original theorem, namely:
	
	\begin{theorem}[Andrieu, Cassaigne; also numbered Theorem \ref{th:Tij_ameliore} in Chapter 2]\label{th:Tij_ameliore_chap3}
		Let $w$ be an infinite $d$-ary  word. Denote by $\Delta_w$ the maximal degree of irrationality of the letters in $w$. Then, for every $n\geq 1$,
		\[
		p_w(n)\geq (\Delta_w-1)(n-1)+d.
		\]
	\end{theorem}
We recall that in the theorem above, the frequencies of the letters in $w$ are not required to exist; instead, we work with the set $F_w$ of their ``pseudo-frequencies'', that is, the set of all subsequential limits of the sequence
	$
(\frac{1}{n}(|\pref_n(w)|_1,\ldots,|\pref_n(w)|_d))_n.
$
The  ``maximal degree of irrationality of the letters'' is then, by definition,
\[
\begin{array}{c}
\ds\Delta_w:=\max\big\{\dim\Span_\Q(f_1,\ldots,f_d) \,|\,(f_1,\ldots,f_d)\in F_w\big\}.
\end{array}
\]

	Our proof relies on linear algebra (notably, the rank--nullity theorem and the study of the dimension of some eigenspaces) applied to a particular rectangular matrix, which we call the ``flow matrix''. Such matrices are closely related to two classical and interconnected objects in combinatorics on words: \emph{Rauzy graphs} and \emph{extension graphs}. By contrast, Tijdeman's original proof is mostly combinatorial. \medskip
	
	Furthermore, the technique developed in this alternative proof enables us to establish the following new result, representing a step forward in the characterization  of $d$-ary words with minimal complexity in Tijdeman's sense.

	\begin{restatable}{theorem}{dendrictheorem}\label{th:dendric_chap3} \emph{(Andrieu, Cassaigne; also numbered Theorem \ref{th:dendric} in Chapter 2)\textbf{.}} Let $d\geq 1$. Every infinite $d$-ary word with rationally independent letter frequencies and complexity $n\mapsto (d-1)n+1$ is dendric.
	\end{restatable}

	The chapter is divided into four sections. In Section \ref{sect:Bigproof_preliminaries}, we  introduce the classical notion of Rauzy graph, and the new notion of flow matrix. In Section \ref{sect:Bigproof_core}, we  prove Theorem \ref{th:Tij_ameliore_chap3}. In Section \ref{sect:comparison_proof_tijdeman}, we compare our proof with that of Tijdeman. Finally, in Section \ref{sect:dendric_proof}, we recall the classical notions of extension graphs and dendricity, and use our proof of Theorem \ref{th:Tij_ameliore_chap3} to establish Theorem \ref{th:dendric_chap3}.

	\section{Preliminaries: definitions of Rauzy graphs and flow matrices}\label{sect:Bigproof_preliminaries}

	\subsection{Rauzy graphs}
	
	Rauzy graphs were introduced, among   other excellent ideas, by Rauzy in \cite{Rau82}. They are a well-studied and useful tool in combinatorics on words.
	By definition, the ``Rauzy graph'' of an infinite word $w$ for a given length $n \in \N$ is the directed graph $\G=(V,E)$ whose:
	\begin{itemize}
		\item vertices are the factors of length $n$ of $w$ ($V:=\lan_n(w)$);
		\item edges are the factors of length $n+1$ of $w$ ($E:=\lan_{n+1}(w)$). Specifically, the edge labeled by $v \in \lan_{n+1}$ joins the word formed by the first $n$ letters of $v$ to the word formed by its last $n$ letters. 
	\end{itemize}
	
	\begin{example}\label{ex:Rauzy_graph}
		The Rauzy graph of the eventually periodic word $$w=2(010)^\omega=2010010010010010...$$ for the length $n=1$ is:
		
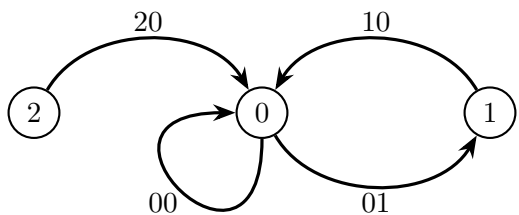
\begin{center}
			\begin{tikzpicture}
				\clip(-3.5,-1.5) rectangle (3.5,1.5);
				\begin{scope}[every node/.style={circle,thick,draw}]
					\node (A) at (0,0) {$0$};
					\node (B) at (3,0) {$1$};
					\node (C) at (-3,0) {$2$};
				\end{scope}
				
				\begin{scope}[>={Stealth[black]},
					every edge/.style={draw=black,very thick}]
					\path [->] (A) edge[in=180, out=-90, looseness=12] (A);
					\path [->] (A) edge[bend right=60] (B);
					\path [->] (B) edge[bend left=-60] (A);
					\path [->] (C) edge[bend right=-60] (A);
				\end{scope}
				
				\draw (-1.3,-1.2) node {$00$};
				\draw (1.5,-1.2) node {$01$};
				\draw (1.5,1.2) node {$10$};
				\draw (-1.5,1.2) node {$20$};
			\end{tikzpicture}
		\end{center}
		Indeed, one readily checks that $\lan_1(w)=\{0,1,2\}$ and $\lan_{2}(w)=\{00,01,10,20\}$.
	\end{example}

	\begin{remark}\label{rk:semi_connected} By definition, a directed graph is ``semi-connected'' if, for every pair of vertices $u$ and $u'$, there exists a path from $u$ to $u'$ or a path from $u'$ to $u$. The Rauzy graph of every infinite word is semi-connected. Indeed, when sliding a window of size $n$ from left to right along $w$, we follow a path along its length-$n$ Rauzy graph that visits every vertex. 
	\end{remark}

	We encourage the reader who is unfamiliar with Rauzy graphs to consider the following two exercises.
	
	\begin{exercise}
		By definition, a directed graph is ``strongly connected'' if  for every pair of vertices $u$ and $u'$, there exists a path going from $u$ to $u'$. \\
		1) Are all Rauzy graphs of all infinite words always strongly connected? \\
		2) Prove that all the Rauzy graphs of a given infinite word $w$ are strongly connected if and only if $w$ is recurrent. \\ 
		3) How can one interpret the \emph{uniform} recurrence of an infinite word $w$ in terms of its Rauzy graphs?
	\end{exercise}

	\begin{exercise} \textcolor{white}{.}\\
		1) Draw the Rauzy graph of the Fibonacci word for the length $4$. (We recall that the Fibonacci word was defined in Chapter 1, Exercise \ref{ex:def_fibo}.) \\
		2) How can one identify the right- and left-special factors of length $4$ on this Rauzy graph? \\
		3) Let $n\in \N$ and $w$ be an arbitrary Sturmian word. Describe the shape of the Rauzy graph of $w$ for length $n$. 
	\end{exercise}

	\subsection{Flow matrices}

	Let $w$ be an infinite word, and $n$ a nonnegative integer.
	By definition, the ``flow matrix'' of $w$ for the length $n$ is  the rectangular matrix $M$, indexed by $u \in \lan_n(w)$ and $ v\in \lan_{n+1}(w)$, and defined by:
	
	\[M_{u,v}:=\begin{cases} 
		1 & \text{ if $v$ starts but does not end with $u$,}\\
		-1 & \text{ if $v$ ends but does not start with $u$,}\\
		0 & \text{ if $v$ neither starts nor ends with $u$, or starts and ends with $u$.}
	\end{cases}
	\]

	Its size is $p_w(n)\times p_w(n+1)$, which is also equal, with the notation of the Rauzy graphs, to $\vert V\vert \times \vert E \vert$. To represent it explicitly, we fix an arbitrary order on the sets $\lan_{n}(w)$ and $\lan_{n+1}(w)$. In this course, we  always choose the lexicographic order. 
	
	\begin{example}\label{ex:flow_matrix}
		Pursuing Example \ref{ex:Rauzy_graph}, the flow matrix, for the length $n=1$, of the eventually periodic word $w=2(010)^\omega=2010010010010010...$ is:
		\[M = \begin{pmatrix}
			0 & 1 & -1 & -1\\
			0 & -1 & 1 & 0\\
			0 & 0 & 0 & 1
		\end{pmatrix}\]
		It is easy to derive the flow matrix from the knowledge of the same-parameter Rauzy graph: on the row indexed by a vertex $u \in \lan_n(w)$, each incoming edge contributes  $-1$, and each outgoing edge contributes $+1$.
	\end{example}

	The name \emph{flow matrix} is inspired by Kirchhoff's junction rule for hydraulic and electric circuits. The flow matrix is different from the classical \emph{adjacency matrix} in graph theory, which is always a square matrix.
	
	\begin{remark}\label{rk:extension_matrix}
		The flow matrix is to be understood as the difference between what could be called the ``right- and left-extension matrices'' of $w$ for the length $n$: $M = R- L$, where 
		\[
		R_{u,v}:=\begin{cases} 1 & \text{ if $v$ starts with $u$,}\\
			0 & \text{ otherwise,}
		\end{cases}
		\]
		is the matrix representing the right-extensions of the length-$n$ factors of $w$ (indeed, the right-extensions of $u\in\lan_n(w)$ are precisely the factors $v\in\lan_{n+1}(w)$ such that $R_{u,v}=1$), and where
		\[
		L_{u,v}:=\begin{cases} 1 & \text{ if $v$ ends with $u$,}\\
			0 & \text{ otherwise,}
		\end{cases}
		\]
		is, similarly, the matrix representing the left-extensions of the length-$n$ factors of $w$.
		\medskip
		
		In particular, it follows from their definitions that:\\
		- Every column of $R$ and of $L$ contains exactly one entry equal to $1$.\\
		- Every row of $R$ contains at least one entry equal to $1$. If $w$ is recurrent, every row of $L$ also contains at least one entry equal to $1$. By contrast, if $w$ is not recurrent, for $n$ sufficiently large, $L$ contains exactly one row full of zeros (as in Example \ref{ex:flow_matrix}). \\
		- A factor $u \in \lan(w)$ is right-special (resp. left-special) if and only if the row indexed by $u$ in the matrix $R$ (resp. $L$) contains several entries equal to $1$. 
	\end{remark}
	
	We are now ready to prove Theorem \ref{th:Tij_ameliore_chap3}.

	
	\section{Proof of Theorem \ref{th:Tij_ameliore_chap3}} \label{sect:Bigproof_core}
	
	\subsection[General structure of the proof of Theorem \ref{th:Tij_ameliore_chap3}]{General structure of the proof}
	
	Let $d\geq 1$. Let $w$ be an infinite $d$-ary word (we recall that this implies $p_w(1)=d$). We  proceed by contradiction and minimality.
	By contradiction, assume that there exists a length $n\geq 1$ such that 
	\begin{equation}\label{eq:ineg_contraposition} p_w(n)< (\Delta_w-1)(n-1)+d,\end{equation}
	where $\Delta_w$ denotes the maximal degree of irrationality of the letters in $w$.
	(Note that for the inequality to hold, we must have $n \geq 2$.)
	Let $m+1$ be the minimal length $n$ for which the inequality \eqref{eq:ineg_contraposition} holds. By minimality of $m+1$, we  have the twofold estimates
	\begin{equation}\label{eq:def_parameter_m}\begin{cases} p_w(m)\geq (\Delta_w-1)(m-1)+d, \\   p_w(m+1)< (\Delta_w-1)m+d \end{cases}\end{equation}
	from which we deduce that
	\begin{equation}\label{eq:jump}
		p_w(m+1)-p_w(m)\leq \Delta_w-2.
	\end{equation}
	
	\bigskip
	By definition of $\Delta_w$, there exists a tuple of real numbers $(f_1,\ldots,f_d)\in [0,1]^d$ (called \emph{pseudo-frequencies of letters}) such that $\dim \Span_{\Q} (f_1, \ldots, f_d) = \Delta_w$, and
	\begin{equation}\label{eq:def_freq_letters}
		\Big(\frac{1}{n_l}(|\pref_{n_l}(w)|_1,\ldots,|\pref_{n_l}(w)|_d)\Big)_l \longrightarrow_{l\rightarrow \infty} (f_1,\ldots,f_d),
	\end{equation}
	for a certain increasing sequence of integers $(n_l)_l$. Up to extracting a subsequence from the sequence $(n_l)_l$, we can also assume that every factor $v$ of length $m+1$ (for the parameter $m$ in the estimates \eqref{eq:def_parameter_m}) admits a pseudo-frequency $f_v$, that is
	\begin{equation}\label{eq:factors_freq} \Big(\frac{|\pref_{n_l}(w)|_v}{n_l} \Big)_{l} \longrightarrow_{l \rightarrow \infty} f_v. \end{equation}

	The goal of the proof is to show that the pseudo-frequencies of length-$m+1$ factors span a rational linear space of dimension at most $\Delta_w-1$, and deduce that, in turn, the pseudo-frequencies of the letters $(f_1,\ldots,f_d)$ also span a rational linear space of dimension at most $\Delta_w-1$. This will contradict the choice made in \eqref{eq:def_freq_letters} for the family of pseudo-frequencies $(f_1, \ldots, f_d)$.
	\medskip
	
	Denote by $M$ the flow matrix of the infinite word $w$ for the specific length $m$. The core of the proof consists in studying the kernel of  $M$, viewed first as a matrix with rational coefficients and then as a matrix with real coefficients. It is divided into the  following three lemmas, which are proved in the next three subsections.

	\begin{lemma}\label{lemma:eigenvector}
		Let $M$ be the flow matrix of the infinite word $w$ for the length $m$. Then the vector of pseudo-frequencies of the length-$m+1$ factors satisfies
		\[ (f_v)_{v \in \lan_{m+1}(w)} \in \ker_{\R}(M).\]
	\end{lemma}
	
	\begin{lemma}\label{lemma:kerM} Let $M$ be the flow matrix of the infinite word $w$ for the length $m$.
		Then the dimension of the kernel of $M$ satisfies
		\[\dim \ker_{\Q}(M)\leq \Delta_w-1. \]
	\end{lemma}
	
	We recall that $\ker_{\R}(M)$ and $\ker_{\Q}(M)$ denote the sets of all vectors $x \in \R^{\lan_{m+1}(w)}$ and $x \in \Q^{\lan_{m+1}(w)}$, respectively, such that $Mx = 0$. They form linear subspaces of $\R^{\lan_{m+1}(w)}$ and $\Q^{\lan_{m+1}(w)}$, respectively. 
	
	\begin{lemma}\label{lemma:dimQ}
		The rational linear space spanned by the pseudo-frequencies of letters satisfies	
		\[\dim \Span_{\Q}(f_1, \ldots, f_d)\leq \Delta_w-1. \]
	\end{lemma}

	If Lemma \ref{lemma:dimQ} holds, then Theorem \ref{th:Tij_ameliore_chap3} follows. Indeed, we have simultaneously $\dim \Span_{\Q}(f_1, \ldots, f_d)= \Delta_w$ (by construction of the pseudo-frequencies of letters $(f_1,\ldots,f_d)$), and $\dim \Span_{\Q}(f_1, \ldots, f_d)< \Delta_w$. A contradiction. 
	
	\subsection{Proof of Lemma \ref{lemma:eigenvector}}

	Let $u \in \lan_m(w)$ and $l \in \N$. 
	Since every occurrence of $u$ in the infinite word $w \in \{1,\dots,d\}^{\N}$, except possibly the first one when $w$ begins with $u$, is preceded by a letter, thus forming a factor $v$ of length $m+1$, we have the general  estimate:
	\begin{equation}\label{eq:pattern_counting_estimate}
	|\pref_{n_l}(w)|_u= \sum_{\substack{v \in \lan_{m+1}(w) \\ \text{ s.t. $v$ ends with $u$}}} |\pref_{n_l}(w)|_v \quad+\quad \begin{cases}
	1 & \text{if $w$ starts with $u$,}\\
	0 & \text{otherwise.}
	\end{cases}
	\end{equation}
	With a similar argument, we obtain the second estimate:
		\begin{equation}\label{eq:pattern_counting_estimate2}
	|\pref_{n_l}(w)|_u= \sum_{\substack{v \in \lan_{m+1}(w) \\ \text{ s.t. $v$ starts with $u$}}} |\pref_{n_l}(w)|_v \quad+\quad \begin{cases}
	1 & \text{if $\pref_{n_l}(w)$ ends with $u$,}\\
	0 & \text{otherwise.}
	\end{cases}
	\end{equation}
	From these estimates and from the definition of the pseudo-frequencies of the length-$m+1$ factors of $w$ \eqref{eq:factors_freq}, we deduce that every factor $u$ of length $m$ also admits a pseudo-frequency $f_u$:
	\[ \Big(\frac{|\pref_{n_l}(w)|_u}{n_l} \Big)_{l} \longrightarrow_{l \rightarrow \infty} f_u,\]
	which satisfies the following conservation law:
	\begin{equation}\label{eq:double_conservation_law} f_u =\sum_{v \in \lan_{m+1}(w) \text{ s.t. $v$ starts with $u$}} f_v = \sum_{v \in \lan_{m+1}(w) \text{ s.t. $v$ ends with $u$}} f_v.\end{equation}

	\begin{remark} In an electric circuit shaped as the Rauzy graph of $w$ for the length $m$, and in which the intensity of the electric current on the branch indexed by $v \in \lan_{m+1}(w)$ is $f_v$, this conservation law is exactly the Kirchhoff's junction rule. \end{remark}

	The second equality in \eqref{eq:double_conservation_law} is trivially equivalent to 
	\[ \sum_{v \in \lan_{m+1}(w) \text{ s.t. $v$ starts with $u$}} f_v - \sum_{v \in \lan_{m+1}(w) \text{ s.t. $v$ ends with $u$}} f_v = 0\]
	and is valid for every $u \in \lan_m(w)$.
	Therefore, by definition of the flow matrix $M$, we have
	\[ M\cdot(f_v)_{v\in \lan_{m+1}(w)} = 0.\]
	Since the vector $(f_v)_{v\in \lan_{m+1}(w)}$ has real entries a priori, we conclude that 
	\[(f_v)_{v\in \lan_{m+1}(w)} \in \ker_{\R}(M).\] 
	Lemma \ref{lemma:eigenvector} is proven.

	\subsection{Proof of Lemma \ref{lemma:kerM}}
	
	Let $M$ be the flow matrix of the infinite word $w$ for the length $m$. We recall that the flow matrix has entries in $\{-1,0,1\}\subseteq \Q$, and has size $p_w(m)\times p_w(m+1)$.

	\medskip
	
	By the rank--nullity theorem applied to $M$ and its transpose, we have
\[
	\begin{cases}\dim\ker_{\Q}(M)=p_{m+1}(w) -\rank(M), \\  \rank(M)=\rank(M^t)=p_m(w)-\dim\ker_{\Q}(M^t); \end{cases}
\]
	from which we derive, using the estimate \eqref{eq:jump}
	\begin{equation}\label{eq:dimensionsMand_transpose}
	\begin{array}{rcl}
		\dim\ker_{\Q}(M)
		&=& p_w(m+1)-p_w(m)+\dim\ker_{\Q}(M^t)\\
		&\leq& \Delta_w-2+\dim\ker_{\Q}(M^t).
	\end{array}
	\end{equation}
	We are now going to prove that $\dim\ker_{\Q}(M^t)=1$, or more precisely, that $\ker_{\Q}(M^t) = \Span_{\Q}((1,\ldots,1)^t)$. Once this is proven, the inequality above becomes
	\[\dim\ker_{\Q}(M) \leq \Delta_w -1,\]
	and establishes Lemma \ref{lemma:kerM}.
	
	\medskip
	
	First, by construction, the sum of the entries in any column of $M$ equals $0$ (it can also be seen as an immediate consequence of Remark \ref{rk:extension_matrix}); therefore, $(1,\ldots,1)^t\in\ker_{\Q}(M^t)$. Conversely, we want to prove that if $x\in\ker_{\Q}(M^t)$, then all the coordinates of $x=(x_u)_{u \in \lan_{m}(w)}$ are equal. (In other words, we want to prove that $\ker_{\Q}(M^t)\subseteq \Span_{\Q}\{(1,\ldots,1)^t\}$.) Let $u \neq u' \in \lan_{m}(w)$. Since the Rauzy graph $\G$ of $w$ for the length $m$ is semi-connected (as is the case of all Rauzy graphs of infinite words, see Remark \ref{rk:semi_connected}), there exists in $\G$ a path from $u$ to $u'$, or from $u'$ to $u$. Say the path is from $u$ to $u'$. Denote by $v_1, \ldots, v_l$ the list of the consecutive edges that lead from $u$ to $u'$, and by $u_0, \ldots, u_{l}$ the list of vertices successively visited, including the endpoints. The path can be represented as follows:
	\begin{equation}\label{eq:path}u = u_0 \underset{v_1}{\to}u_1 \underset{v_2}{\to} u_2 \underset{v_3}{\to} \cdots \to u_{l-1} \underset{v_l}{\to} u_l = u'.\end{equation}	Since $u'\neq u$, we may assume, after possibly shortening the path, that the vertices  $u_0,\ldots,u_{l} $ are pairwise distinct.

	Now, by definition of a Rauzy graph, for every $i \in \{1,\ldots,l\}$, the factor $v_i$ starts with $u_{i-1}$ and ends with $u_{i}$. Therefore,  one finds exactly two nonzero entries in the row of $M^t$ indexed by the factor $v_i$:  an entry ``$+1$'' at position $u_{i}$, and an entry ``$-1$'' at position $u_{i-1}$, which are distinct. Finally, since $x\in\ker_{\Q}(M^t)$, we have $x_{u_i}-x_{u_{i-1}}=0$.
	The equality $x_{u_{i-1}} = x_{u_i}$ being true for every $i \in \{1,\ldots,l\}$, it propagates along the path \eqref{eq:path}, and we obtain the desired equality $x_u = x_{u'}$. We thus proved that all the coordinates of $x \in \ker_{\Q}(M^t)$ are equal, from which we deduce that $\ker_{\Q}(M^t)$ is contained in, and in fact equal to, $\Span_{\Q}((1,\ldots,1)^t)$. Therefore, $\dim \ker_{\Q}(M^t) = 1$, and the equality \eqref{eq:dimensionsMand_transpose} becomes
	\[	\dim\ker_{\Q}(M) \leq \Delta_w-1.\] Lemma \ref{lemma:kerM} is proven.
	
	\subsection{Proof of Lemma \ref{lemma:dimQ}}
	
	\emph{For this proof, we advise the reader to keep track of the class (scalar, vector with rational or real entries, etc.) of each object.}
	
	\bigskip

	Denote $\delta=\dim\ker_{\Q}(M)$ (by Lemma \ref{lemma:kerM} we already know that $\delta \leq \Delta_w -1$). Let $b_1,\ldots,b_\delta\in\Q^{\lan_{m+1}(w)}$ be a basis of $\ker_{\Q}(M)$, that is, a basis of the kernel of $M$ viewed as a matrix with \emph{rational entries}.
	The key point of the proof is the following general result from linear algebra: since $\Q \subseteq \R$, the family of vectors $b_1,\ldots,b_\delta$ also forms a basis of $\ker_{\R}(M)$, that is, of $\ker(M)$ viewed as a matrix with \emph{real entries.} (For completeness, we recall this general result in the next subsection.) Therefore, since, by Lemma \ref{lemma:eigenvector}, the pseudo-frequency vector $f:=(f_{v})_{v \in \lan_{m+1}(w)}$ belongs to $\ker_\R(M)$, there exist real numbers $\lambda_1,\ldots,\lambda_\delta\in\R$ such that \[f=\sum_{i=1}^{\delta}\lambda_i b_i.\]
	 This vector equality is equivalent to the following $p_{m+1}(w)$ scalar equalities: for every $v \in \lan_{m+1}(w)$, we have
	\[f_{v}=\sum\limits_{i=1}^{\delta}\lambda_i b_{i,v}\]
	where $b_{i,v}$ denotes the coordinate of the vector $b_i$ indexed by $v$.
	In this equality, we have $f_v \in \R$, $\lambda_i \in \R$, and $b_{i,v} \in \Q$. Therefore, the real number
	$f_v$ belongs to the rational linear space $\Span_\Q(\lambda_1,\ldots,\lambda_\delta)$; and since the $\lambda_i$ do not depend on the choice of $v\in\lan_{m+1}(w)$, it follows that 
	\[\Span_\Q((f_v)_{v\in \lan_{m+1}(w)}) \subseteq\Span_\Q(\lambda_1,\ldots,\lambda_\delta).\]

	We now come back to the pseudo-frequencies $f_1,\ldots,f_d$ of the \emph{letters} in $w$. With a counting argument similar to \eqref{eq:pattern_counting_estimate}, it is easy to establish that
	\[ f_j =\sum_{v \in \lan_{m+1}(w) \text{ s.t. $v$ starts with $j$}} f_v\; ,\]
	for every $j \in \{1,\ldots,d\}$; from which we deduce that
	\[f_j \in \Span_\Q((f_v)_{v\in \lan_{m+1}(w)}) \]
	(still for every $j \in \{1,\ldots,d\}$), hence
	\[\Span_\Q(f_1,\ldots,f_d)\subseteq \Span_\Q((f_v)_{v\in \lan_{m+1}(w)}) \subseteq\Span_\Q(\lambda_1,\ldots,\lambda_\delta).\] 
	The above inclusions of linear spaces imply the inequality
\[\dim \Span_\Q(f_1,\ldots,f_d)\leq \dim \Span_\Q(\lambda_1,\ldots,\lambda_\delta)\leq \delta.\]
	Since $\delta =\dim\ker_{\Q}(M)\leq \Delta_w-1$ by Lemma \ref{lemma:kerM}, Lemma \ref{lemma:dimQ} is proven. The proof of Theorem \ref{th:Tij_ameliore_chap3} is complete.
	
	\subsection{Annex: A brief reminder of linear algebra}\label{ssect:reminder_linear_algebra}
	
	For the reader's convenience, we prove the following general lemma, which played a crucial role in the proof of Lemma \ref{lemma:dimQ}. 
	
	\begin{lemma}\label{lemma:linear_algebra}
		\emph{(i)} If the vectors $e_1,\ldots,e_d  \in \Q^n$ are $\Q$-linearly independent, then they are also $\R$-linearly independent.\\
		\emph{(ii)} Let $M\in\mathcal{M}_{m,n}(\Q)$. If the vectors $e_1,\ldots,e_d \in \Q^n$ form a basis of $\ker_\Q(M)$, then they also form a basis of $\ker_\R(M)$.
	\end{lemma}
	
	\begin{proof}
		(i) First, suppose $d=n$. In this case, the $\Q$-linearly independent vectors $e_1,\ldots,e_n$ form a basis of $\Q^n$ and therefore generate the canonical basis of  $\Q^n$ through rational linear combinations. Since the canonical basis of $\Q^n$ coincides with the canonical basis of $\R^n$, it follows that every vector in $\R^n$ can be expressed as a linear combination with real coefficients of the vectors
 $e_1,\ldots,e_n$. Hence, $e_1,\ldots,e_n$ form a basis of $\R^n$.
		
		Now, in the general case $d\leq n$, we complete the set $\{e_1, \ldots, e_d\}$ of $\Q$-linearly independent vectors into a basis $\{e_1, \ldots, e_n\}$ of $\Q^n$. By the argument above, this extended set forms a basis of $\R^n$, from which it follows  that the vectors $e_1,\ldots,e_d$ are $\R$-linearly independent.
		\medskip
		
		(ii) By the previous point, the family $(e_1,\ldots,e_d)$  is $\R$-linearly independent. To show that it forms a basis of $\ker_{\R}(M)$, it suffices to verify that $\dim \ker_{\R}(M) = d$. Applying the rank--nullity theorem gives
		\[\dim \ker_{\R}(M) = n - \rank(M); \] 
		and the rank of $M$ is the same whether $M$ is considered over $\Q$ or $\R$. (This can be verified by observing that Gaussian elimination is unaffected by the ambient field.) Hence, applying rank--nullity  over $\Q$ yields
		\[\dim \ker_{\Q}(M) = n - \rank(M), \] 
		and therefore
		$\dim \ker_{\R}(M) = \dim \ker_{\Q}(M) = d$.
		This proves that $(e_1,\ldots,e_d)$ forms a basis of $\ker_{\R}(M)$, completing the proof of Lemma \ref{lemma:linear_algebra}.\end{proof}

	
	\section{Comparison with Tijdeman's original proof}\label{sect:comparison_proof_tijdeman}
	
	Both Tijdeman's original proof and our alternative proof proceed by contradiction, and rely on a thorough study of the Rauzy graph of $w$ for a specific length $m$, defined as the smallest integer for which the first difference of the complexity function is smaller than expected. However: 
	\begin{itemize}
		\item our proof is purely algebraic, and crucially relies on the notion of flow matrix, and on the following new result.
		
		\begin{theorem}[Andrieu, Cassaigne]
			Let $d \geq 1$ and let $w$ be an infinite $d$-ary word. Let $n \in \N$, and denote by $M$ the flow matrix of $w$ for length $n$. Let $K$ be either $\Q$ or $\R$. Then we have
			\begin{enumerate}
			\item $\ker_K (M^t) = \Span((1,\ldots,1)^t)$; 
			\item $\dim \ker_K (M) = p_w(n+1)-p_w(n)+1$.
			\end{enumerate}
			
			Furthermore, if $(f_v)_{v\in\lan_{n+1}(w)}$ denotes the frequencies of the factors of length $n+1$ in $w$, or any subsequential limit of  \[ \Big( \Big( \frac{|\pref_{l}(w)|_v}{l}  \Big)_{v \in \lan_{n+1}(w)}\Big)_{l}, \] then
		 $(f_v)_{v\in\lan_{n+1}(w)} \in \ker_K (M)$.
		\end{theorem}

		By contrast, Tijdeman's original proof does not use linear algebra. Instead, it relies on combinatorial arguments and a notion of ``$p$-passing number'' in directed graphs.
		\item The original result from Tijdeman is slightly weaker than our Theorem \ref{th:Tij_ameliore_chap3}: his estimate on the growth of the complexity function of an infinite word $w$ relies on the \emph{minimal} rational degree of the letters in $w$, while ours relies on their \emph{maximal} rational degree (see the discussion in Chapter 2, Section \ref{ssect:tijdeman_complet_t_ameliore}.) However, we think that Tijdeman's original proof can be restructured in order to establish our stronger result. 
	\end{itemize}

	\section{A byproduct of our proof: words with Tijdeman's minimal complexity are dendric}\label{sect:dendric_proof}

	The aim of this last section is to prove, with the ingredients of our alternative proof of Tijdeman's theorem, the following new result.
	
	\dendrictheorem*

	This theorem is a step forward towards understanding the combinatorial structure of infinite words with rationally independent letter frequencies and minimal complexity, see the complete discussion and open questions in Chapter 2, Section \ref{sect:dary_min_complexity_characterization}. 
	
	\medskip

	We start by recalling the definition of dendric words.

	An infinite word is \emph{dendric} if the extension graph of each of its factors is a tree. We recall that an undirected graph is a \emph{tree} when it is \emph{connected} (by definition: when there exists a path between every two vertices), and \emph{acyclic} (by definition: it contains no cycle). The \emph{extension graph} of a factor $u \in \lan_n(w)$ is the undirected bipartite graph $\ext(u)=(L \sqcup R, B)$ defined as follows: \begin{itemize}
	\item We denote by $L$ the set of all letters $a$ such that $au \in \lan_{n+1}(w)$. Symmetrically, we denote by $R$  the set of all letters $b$ such that $ub \in \lan_{n+1}(w)$;
	\item There is an edge between two vertices $a \in L$ and $b \in R$ if and only if $aub \in\lan_{n+2}(w)$. The set of all edges is denoted by $B$.
	 \end{itemize}
	The notion of dendric words was introduced in \cite{BDDLP15}, whereas the concept of extension graphs can be traced back to \cite{Cas97}.

	\begin{remark} Extensions graphs and Rauzy graphs are intimately related. More precisely, to construct the Rauzy graph $\mathcal{G}(n+1)$ of a recurrent word $w$ for length $n+1$, it suffices to know  its Rauzy graph $\mathcal{G}(n)$ for length $n$, and the extension graphs of all its length-$n$ factors that are simultaneously left- and right-special. (Note that such factors are represented in $\mathcal{G}(n)$ by a vertex with several incoming and several outgoing edges.)
		\end{remark}
	

	\begin{example}The infinite word $w_1=11111111...$ is dendric. Indeed, the extension graphs of its factors are all equal to the bipartite graph in Figure \ref{fig:extension_graphs_ex}, left.
	By contrast, the periodic infinite word $w_2=112211221122...$ is not dendric. Indeed, the extension graph of the empty word $\eps \in \lan(w_2)$ is the non-acyclic graph on Figure \ref{fig:extension_graphs_ex}, right.
	\end{example}

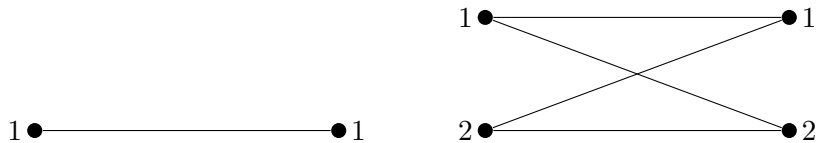
\begin{figure}
	\begin{center}
		\begin{tikzpicture}[
		every node/.style={circle, fill=black, inner sep=2pt},
		]
		\node (L1) at (0,0) {};
		\node (R1) at (4,0) {};
		
		\node[fill=none, left] at (L1) {1};
		\node[fill=none, right] at (R1) {1};
		
		\draw (L1) -- (R1);
			\end{tikzpicture} \qquad
			\begin{tikzpicture}[
			every node/.style={circle, fill=black, inner sep=2pt},
			]
			
			\node (L1) at (0,1.5) {};
			\node (L2) at (0,0) {};
			
			\node (R1) at (4,1.5) {};
			\node (R2) at (4,0) {};
			
			\node[fill=none, left]  at (L1) {1};
			\node[fill=none, left]  at (L2) {2};
			\node[fill=none, right] at (R1) {1};
			\node[fill=none, right] at (R2) {2};
			
			\draw (L1) -- (R1);
			\draw (L1) -- (R2);
			\draw (L2) -- (R1);
			\draw (L2) -- (R2);
			
			\end{tikzpicture}
	\end{center}
\caption{\label{fig:extension_graphs_ex} The extension graph of every factor in $w_1=11111...$ (left), and of the empty word in $w_2=112211221122...$ (right).}
\end{figure}

		\begin{exercise} [Dendricity and recurrence]\textcolor{white}{.}\\
		1) Prove that the infinite word $w_1 = 21111111...$ is not dendric. \\
			2) Prove that every  dendric infinite word is recurrent.\\
				3) A word is said to be \emph{bi-infinite} if its  letters are indexed by $\Z$ rather than $\N$. Prove, by contrast, that the non-recurrent bi-infinite word $w_2 = ...1112111...$ is dendric.
				
	\end{exercise}

	\begin{exercise} Prove that every Sturmian word is dendric. \emph{Indication}: use Proposition \ref{prop:basics}, Chapter 1.
\end{exercise}
	
 It is well known that all $d$-ary dendric words have complexity $n\mapsto (d-1)n+1$  \cite{BDDLP15}.	Interestingly, dendric words are even characterized by their complexity, provided solely that their extension graphs are  connected.
	
	\begin{proposition}\label{prop:charac_dendric}
	An infinite $d$-ary word $w$ is dendric if and only if its complexity is $n\mapsto (d-1)n+1$, and the extension graphs of all its factors are connected.
\end{proposition}

\begin{proof} Let $G$ be an undirected graph with $v$ vertices and $e$ edges. We recall two classical results from graph theory:\\
	(1) if $G$ is connected, then one has $e\geq v - 1$; \\
	(2) $G$ is a nonempty tree if and only if $G$ is connected and $e= v - 1$.

Let $w$ be an infinite word.
Denote by $\ext(u)=(L(u)\sqcup R(u),B(u))$ the extension graph of any of its factors $u$. First, observe that the discrete derivative of the complexity function of $w$ is
\[p(n+1)-p(n) = \sum_{u \in \lan_n(w)} (\vert R(u)\vert -1),\]
for all $n\in N$. Similarly, we have
\[p(n+2)-p(n+1) = \sum_{v \in \lan_{n+1}(w)} (\vert R(v)\vert -1) = \sum_{u \in \lan_n(w)} \sum_{a \in L(u)} (\vert R(au)\vert -1).\]
Combining these two identities, we obtain that the second derivative of the complexity $\Delta(n):= (p(n+2)-p(n+1))-(p(n+1)-p(n))$ is equal\footnote{It is interesting, but not useful in this lecture, to remark that the identity \eqref{eq:second_derivative} holds for every \emph{factorial language} $\lan$, without assumption on its \emph{extendability}.} to
\begin{equation}\label{eq:second_derivative}
\begin{array}{ll}\Delta(n)  & = \sum_{u \in \lan_n(w)} \sum_{a \in L(u)} (\vert R(au)\vert -1) - \sum_{u \in \lan_n(w)}  (\vert R(u)\vert-1)\\
& = \sum_{u \in \lan_n(w)} (\vert B(u)\vert - \vert L(u) \vert - \vert R(u) \vert + 1)\end{array}\end{equation} for every $n \in \N$. 

\medskip

Now, if $w$ is a $d$-ary dendric word, then by (2) each contribution in the sum \eqref{eq:second_derivative} is equal to $0$, hence the second derivative of the complexity is $0$. The complexity is thus an affine function. Since $p(0)=1$ and $p(1)=d$, we conclude that $p(n)=(d-1)n+1$ for all $n\in \N$, and the direct implication in Proposition \ref{prop:charac_dendric} is proved.

We prove the converse implication. On the one hand, since the complexity of $w$ is $p(n)=(d-1)n+1$, its second derivative is $0$. On the other hand, by (1), since the extension graphs of all factors of $w$ are connected, every contribution $ \vert B(u)\vert - \vert L(u) \vert - \vert R(u) \vert + 1$ in the sum \eqref{eq:second_derivative} is nonnegative. Therefore, for every $u \in \lan_n(w)$ and every $n \in \N$, we have 
\[ \vert B(u)\vert - \vert L(u) \vert - \vert R(u) \vert + 1 = 0.\] By (2), we conclude that $w$ is dendric. The proof of Proposition \ref{prop:charac_dendric} is complete.
\end{proof}

\begin{exercise}
Similarly to Proposition \ref{prop:charac_dendric}, prove that an infinite $d$-ary word $w$ is dendric if, and only if, its complexity is $n\mapsto (d-1)n+1$, and the extension graphs of all its factors are acyclic.
\end{exercise}

	With the characterization of Proposition \ref{prop:charac_dendric} in mind, we can prove Theorem \ref{th:dendric_chap3}. 
	

\begin{proof}[Proof of Theorem \ref{th:dendric_chap3}] Let $d \geq 1$. Let $w$ be an infinite $d$-ary word with rationally independent letter frequencies and complexity $n\mapsto (d-1)n+1$. 
	By Proposition \ref{prop:charac_dendric}, to prove that $w$ is dendric, it suffices to show that the extension graphs of all its factors are connected. We proceed by contradiction, and assume that there exists a factor $u$ of $w$ whose extension graph $\ext(u)=(R\sqcup L, B)$ is not connected. 
	The vertices of $\ext(u)$ can thus be partitioned into two nonempty sets $V_1,V_2$ such that there exists no edge between a vertex in $V_1$ and a vertex in $V_2$. Denote by $L_1,L_2,R_1,R_2$ the subpartition $L_1:=L\cap V_1$, $L_2:=L\cap V_2$, $R_1:=R\cap V_1$ and $R_2:=R\cap V_2$. It follows from the nonemptyness of $V_1, V_2$ that at least three of the subsets $L_1,L_2,R_1,R_2$ are nonempty. (Although it is not important in the proof, note that the recurrence of $w$, which was proved in Chapter 2, Proposition \ref{prop:tijdeman_recurrent}, implies that the four subsets $L_1,L_2,R_1,R_2$ are nonempty.)
\medskip

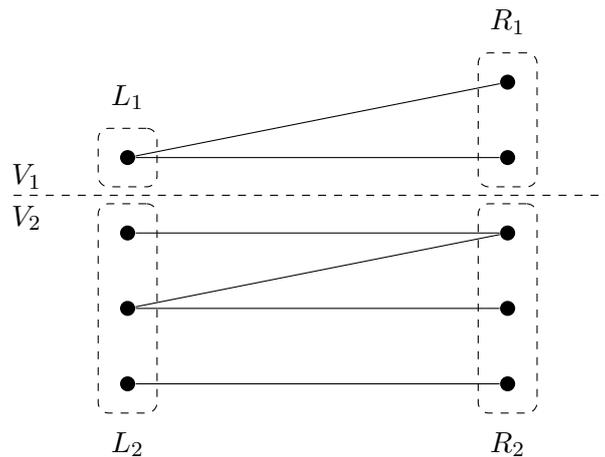
\begin{figure}[h!]
	\begin{center}
		\begin{tikzpicture}[
		vertex/.style={circle, fill=black, inner sep=2pt},
		group/.style={draw, dashed, rounded corners, inner sep=8pt}
		]
		
		
		\node[vertex] (L1) at (0,3) {};
		
		\node[vertex] (L2) at (0,2) {};
		\node[vertex] (L3) at (0,1) {};
		\node[vertex] (L4) at (0,0) {};
		
		
		\node[vertex] (R1) at (5,4) {};
		\node[vertex] (R2) at (5,3) {};
		
		\node[vertex] (R3) at (5,2) {};
		\node[vertex] (R4) at (5,1) {};
		\node[vertex] (R5) at (5,0) {};
		
		
		\draw (L1) -- (R1);
		\draw (L1) -- (R2);
		
		\draw (L2) -- (R3);
		\draw (L3) -- (R3);
		\draw (L3) -- (R4);
		
		\draw (L4) -- (R5);
		
		\draw[dashed] (-1.5,2.5) -- (6.5,2.5);
		
		
		\node[group, fit=(L1)] {};
		\node[group, fit=(R1)(R2)] {};
		\node[group, fit=(L2)(L3)(L4)] {};
		\node[group, fit=(R3)(R4)(R5)] {};
		
		
		\node[fill=none, left] at (-1,2.75) {$V_1$};
		\node[fill=none, left] at (-1,2.2) {$V_2$};
		
		\node[fill=none] at (0,3.8) {$L_1$};
		\node[fill=none] at (5,4.8) {$R_1$};
		\node[fill=none] at (0,-0.8) {$L_2$};
		\node[fill=none] at (5,-0.8) {$R_2$};
		
		\end{tikzpicture}
	\end{center}
	\caption{\label{fig:extensiongraph_partition} Example of a partition of the extension graph.}
\end{figure}

	Denote by $\mathcal{G}$ the Rauzy graph of $w$ for the length $n = \vert u \vert$. We are going to modify $\mathcal{G}$ into a new graph $\mathcal{G'}$ by splitting the vertex $u$ into two copies $u_1$ and $u_2$, and distributing its incoming and outgoing edges between  $u_1$ and $u_2$ according to the partition of its extension graph $\ext(u)$. More precisely, we define the graph $\mathcal{G'}=(V',E')$ as follows.\\
	\hspace*{0.2cm}	- Its vertices are $V'=\lan_n(w)\backslash\{u\}\cup\{u_1,u_2\}$.\\
	\hspace*{0.2cm} - All edges in $\mathcal{G}$ that are not adjacent to the vertex $u$  are preserved in $\mathcal{G'}$. The incoming edges to $u_1$ (respectively, $u_2$) in $\mathcal{G'}$ are the incoming edges to $u$ in $\mathcal{G}$ that furthermore belong to  $L_1\cdot u=\{au, a \in L_1\}$ (respectively, to $L_2\cdot u$). Similarly, the outgoing edges from $u_1$ (respectively, $u_2$) in $\mathcal{G'}$ are the outgoing edges from $u$ in $\mathcal{G}$ that furthermore belong to $u\cdot R_1$ (respectively, $u\cdot R_2$).
	\medskip 
	
	The graph $\mathcal{G'}$ thus defined has $\vert V'\vert = p_w(n)+1$ vertices, and $\vert E'\vert=\vert E \vert=p_w(n+1)$ edges. Although $\mathcal{G'}$ is not a Rauzy graph anymore, it inherits many of its properties. Notably, $\mathcal{G'}$ is still semi-connected. 
	
	\begin{lemma}\label{lemma:connectedness_Gprime}
The graph $\mathcal{G'}$ is semi-connected. 
	\end{lemma}

\begin{proof}
The proof is a slight modification of that of Lemma \ref{lemma:kerM}. Since the edges of $\mathcal{G}'$ are exactly the factors of length $n+1$ in the infinite word $w$, there always exists in $\mathcal{G}'$ a path that connects any two edges $v_1$ and $v_2 \in \lan_{n+1}(w)$. (Such a path is for example given by sliding a window between an occurrence of $v_1$ and an occurrence of $v_2$ in $w$.) It now suffices to observe that every vertex in $\mathcal{G}'$ is attached to at least one edge. This is trivially true for all vertices in $V'\backslash \{u_1,u_2\}$ which all admit an outgoing edge. This is also true for $u_1$ and $u_2$ due to the fact that at least three of the four subsets $L_1,L_2,R_1,R_2$ are nonempty. Lemma \ref{lemma:connectedness_Gprime} is proved.
	\end{proof}
	 
	Furthermore, $\mathcal{G}'$ admits a ``flow matrix'' $M'$ of size $\vert V'\vert \times \vert E'\vert$ defined as follows: the entry at position $(u,v)\in V'\times E'$ is $1$ if the oriented edge $v$ departs from $u$ and does not arrive at $u$, $-1$ if $v$ arrives at $u$ but does not depart from $u$, and $0$ otherwise. We  claim that the core lemmas of our proof of Tijdeman's theorem (namely Lemmas \ref{lemma:eigenvector} and \ref{lemma:kerM}) remain true for $M'$, and that their proofs  require  only minor adaptation.
	
	\begin{lemma}\label{lemma:three_lemmas_adapted} With the notation above, we have
		\begin{enumerate}
			\item $\dim \ker_{\Q}(M') = \vert E'\vert - \vert V'\vert +1 = d-1$ ,
			\item the pseudo-frequency vector of length-$n+1$ factors of $w$  belongs to the kernel of $M'$, \emph{i.e.},
			\[M'\cdot(f_v)_{v\in \lan_{n+1}(w)} =0.\]
		\end{enumerate}
	\end{lemma}

	\begin{proof}[Proof of Lemma \ref{lemma:three_lemmas_adapted}]The proof of the first assertion is an immediate adaptation of that of Lemma \ref{lemma:kerM}. We recall that it relies on the connectedness of $\mathcal{G'}$ (proven in Lemma \ref{lemma:connectedness_Gprime}), from which we establish that $ \ker_{\Q} (M'^t) = \Span((1,\ldots,1)^t)$. Then, a double application of the rank-nullity theorem yields the desired equality. We prove the second assertion, which can still be understood as a Kirchhoff's junction rule on $\mathcal{G'}$. Denote by $M'[r]$ the row of $M'$ indexed by $r \in V'$. Our aim is to prove that for every $r \in V'$,
		\begin{equation}\label{eq:row_zero}
		M'[r]\cdot (f_v)_{v \in \lan_{n+1}(w)} = 0.
		\end{equation}
		This equation is trivially satisfied for every $r\neq u_1,u_2$ since $M'[r]$ is also a row of $M$, the flow matrix of the original Rauzy graph $\mathcal{G}$ for length $n$. It thus suffices to reuse Lemma \ref{lemma:eigenvector}.
		For $r = u_1$ and $r =u_2$, the equation \eqref{eq:row_zero} comes from the following equalities:
		for every $a \in \{1,\ldots,d\}$, we have
		\[f_{au} = \sum_{b \in \{1,\ldots,d\}} f_{aub}; \]  and symmetrically, for every $b \in \{1,\ldots,d\}$, we have
		\[f_{ub} = \sum_{a \in \{1,\ldots,d\}} f_{aub}.\]
		Since there exists no edge that connects $V_1$ and $V_2$, we have that  $f_{aub} = 0$ if $(a,b) \in L_1\times R_2 \cup L_2\times R_1$. Hence
		\begin{equation}\label{eq:refinedkirkof1} \sum_{a \in L_1} f_{au} =  \sum_{a \in L_1}\sum_{b \in R_1} f_{aub}=\sum_{b \in R_1}f_{ub},\end{equation}
		and 
		\begin{equation}\label{eq:refinedkirkof2} \sum_{a \in L_2} f_{au} =  \sum_{a \in L_2}\sum_{b \in R_2} f_{aub} = \sum_{b \in R_2}f_{ub}.\end{equation}
		The equations \eqref{eq:refinedkirkof1} and \eqref{eq:refinedkirkof2} are respectively equivalent to $
		M'[u_1]\cdot(f_v)_{v \in \lan_{n+1}(w)} =0$ and $M'[u_2]\cdot(f_v)_{v \in \lan_{n+1}(w)}=0$. Lemma \ref{lemma:three_lemmas_adapted} is proven.
	\end{proof}

	From Lemma \ref{lemma:three_lemmas_adapted}, with the exact same calculation as in Lemma \ref{lemma:dimQ}, we prove that the letter frequencies of $w$ span a rational linear space of dimension at most $d-1$. We reach the desired contradiction. The proof of Theorem \ref{th:dendric_chap3} is complete.
\end{proof}

	%
	%

\cleardoublepage

	\bibliography{biblio_ValpoLecture}
	\bibliographystyle{alpha}
	
\end{document}